\documentclass[11pt]{amsart}   	
\usepackage[margin=1in]{geometry}
\usepackage{subfigure}					
\usepackage{amsmath}
\usepackage{amssymb, amsfonts}
\usepackage[nobysame]{amsrefs}
\usepackage{amsthm}
\usepackage{hyperref}
\usepackage{mathtools}
\usepackage{tikz}
\usetikzlibrary{cd}
\usetikzlibrary{decorations.markings}
\usepackage[shortlabels]{enumitem}
\usepackage{comment}

\usepackage[noabbrev,capitalise,nameinlink]{cleveref}

\definecolor{lavender}{rgb}{0.4,0,1.0}

\hypersetup{colorlinks=true, citecolor=lavender, linkcolor=lavender,urlcolor=lavender}

\newtheorem{theorem}{Theorem}[section]
\newtheorem{proposition}[theorem]{Proposition}
\newtheorem{corollary}[theorem]{Corollary}

\newtheorem{lemma}[theorem]{Lemma}

\theoremstyle{definition}

\newtheorem{remark}[theorem]{Remark}
\newtheorem{example}[theorem]{Example}

\newcommand{\s}{\mathtt{s}}

\newcommand{\LL}{\mathcal{L}}

\newcommand{\Orn}{\mathrm{Orn}}
\newcommand{\PT}{\mathrm{PT}}
\newcommand{\MN}{\mathrm{MN}}
\newcommand{\TT}{\mathsf{T}}
\newcommand{\NN}{\mathcal{N}}
\newcommand{\VV}{\mathsf{V}}
\newcommand{\Broom}{\mathsf{Broom}}
\newcommand{\wo}{w_{\circ}}
\newcommand{\Weak}{\mathrm{Weak}}
\newcommand{\OO}{\mathcal{O}}
\newcommand{\Inv}{\mathrm{Inv}}

\newcommand{\dfn}[1]{\textcolor{blue}{\emph{#1}}}

\title{Operahedron Lattices}

\author[]{Colin Defant}
\address[]{Department of Mathematics, Harvard University, Cambridge, MA 02139, USA}
\email{{\href{mailto:colindefant@gmail.com}{colindefant@gmail.com}}}

\author[]{Andrew Sack}
\address[]{Department of Mathematics, University of California, Los Angeles, CA 90095, USA}
\email{{\href{mailto:andrewsack@math.ucla.edu}{andrewsack@math.ucla.edu}}}

\begin{document}

%\keywords{associahedron, permutohedron, operahedron, lattice, semidistributive} 

\begin{abstract}
Laplante-Anfossi associated to each rooted plane tree a polytope called an \emph{operahedron}. He also defined a partial order on the vertex set of an operahedron and asked if the resulting poset is a lattice. We answer this question in the affirmative, motivating us to name Laplante-Anfossi's posets \emph{operahedron lattices}. The operahedron lattice of a chain with $n+1$ vertices is isomorphic to the $n$-th Tamari lattice, while the operahedron lattice of a claw with $n+1$ vertices is isomorphic to $\mathrm{Weak}(\mathfrak S_n)$, the weak order on the symmetric group $\mathfrak S_n$. We characterize semidistributive operahedron lattices and trim operahedron lattices. Let $\Delta_{\mathrm{Weak}(\mathfrak S_n)}(w_\circ(k,n))$ be the principal order ideal of $\mathrm{Weak}(\mathfrak S_n)$ generated by the permutation ${w_\circ(k,n)=k(k-1)\cdots 1(k+1)(k+2)\cdots n}$. Our final result states that the operahedron lattice of a broom with $n+1$ vertices and $k$ leaves is isomorphic to the subposet of $\mathrm{Weak}(\mathfrak S_n)$ consisting of the preimages of $\Delta_{\mathrm{Weak}(\mathfrak S_n)}(w_\circ(k,n))$ under West's stack-sorting map; as a consequence, we deduce that this subposet is a semidistributive lattice.  
\end{abstract}

\maketitle

\section{Introduction}\label{sec:intro}

Let $\PT_n$ denote the set of rooted plane trees with $n+1$ vertices. We view a tree $\TT\in\PT_n$ with vertex set $\VV$ as a poset $(\VV,\leq_\TT)$, where the partial order $\leq_\TT$ is defined so that every non-root vertex covers exactly one element. Thus, the root vertex is the unique minimal element of $\TT$. A \dfn{tube} of $\TT$ is a set of vertices that induces a connected subgraph of $\TT$. Every tube $\tau$ has a unique minimal element under $\leq_\TT$ that we call the \dfn{root} of $\tau$. Two sets are said to be \dfn{nested} if one is contained in the other. A \dfn{nesting} of $\TT$ is a collection $\mathcal N$ of tubes of $\TT$ such that 
\begin{itemize}
\item $\VV\in\NN$;
\item each tube in $\NN$ has cardinality at least $2$; 
\item any two tubes in $\mathcal N$ are either nested or disjoint.
\end{itemize} Let ${\bf N}(\TT)$ be the poset of nestings of $\TT$, ordered by containment, and let $\widehat{\bf N}(\TT)$ be the poset obtained from ${\bf N(\TT)}$ by appending a new element that is greater than every element of ${\bf N(\TT)}$. 

Let $\TT\in\PT_n$. We say $\TT$ is a \dfn{chain} if $\leq_\TT$ is a total order. We say $\TT$ is a \dfn{claw} if every non-root vertex of $\TT$ covers the root of $\TT$. Let $\Broom_{k,n}\in\PT_n$ denote the rooted plane tree obtained by identifying the unique maximal element of the chain in $\PT_{n-k+1}$ with the root of the claw in $\PT_{k+1}$; we call $\Broom_{k,n}$ a \dfn{broom}. See \cref{fig:chain_claw_broom}.

\begin{figure}[ht]
  \begin{center}\includegraphics[height=3.854cm]{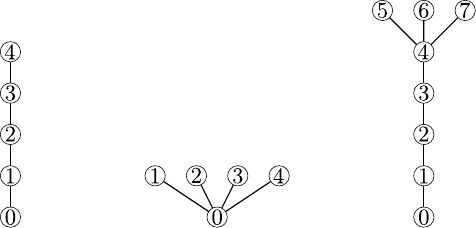}
  \end{center}
  \caption{On the left is the chain in $\PT_4$. In the middle is the claw in $\PT_4$. On the right is the broom $\Broom_{3,7}\in\PT_7$. We have identified the vertex set of each tree in $\PT_n$ with $\{0,1,\ldots,n\}$ in a manner such that $0,1,\ldots,n$ is the preorder traversal of the tree. }\label{fig:chain_claw_broom}
\end{figure}

In a recent breakthrough, Masuda, Thomas, Tonks, and Vallette \cite{Masuda} found coherent cellular approximations of the diagonals of associahedra. This allowed them to define a topological cellular operad structure on the Loday realizations of associahedra. Motivated by these results, Laplante-Anfossi \cite{laplante2022diagonal} defined an \dfn{operahedron} of a rooted plane tree $\TT$ to be a polytope whose face lattice is isomorphic to the dual of $\widehat {\bf N}(\TT)$. Operahedra are special examples of \emph{poset associahedra}, which Galashin recently introduced using tubings of an arbitrary poset \cite{galashin2021poset}. 
Operahedra simultaneously generalize associahedra and permutohedra; indeed, associahedra are just operahedra of chains, while permutohedra are operahedra of claws. Laplante-Anfossi constructed Loday realizations of operahedra and found coherent cellular approximations of their diagonals. He then defined a topological cellular operad structure on these Loday realizations. 

\begin{figure}
\begin{center}\includegraphics[height=8.933cm]{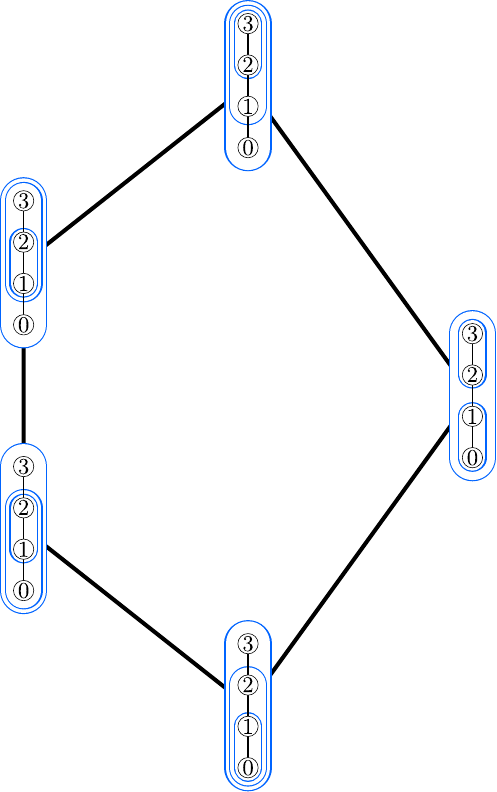}\qquad\qquad\qquad\qquad\includegraphics[height=8cm]{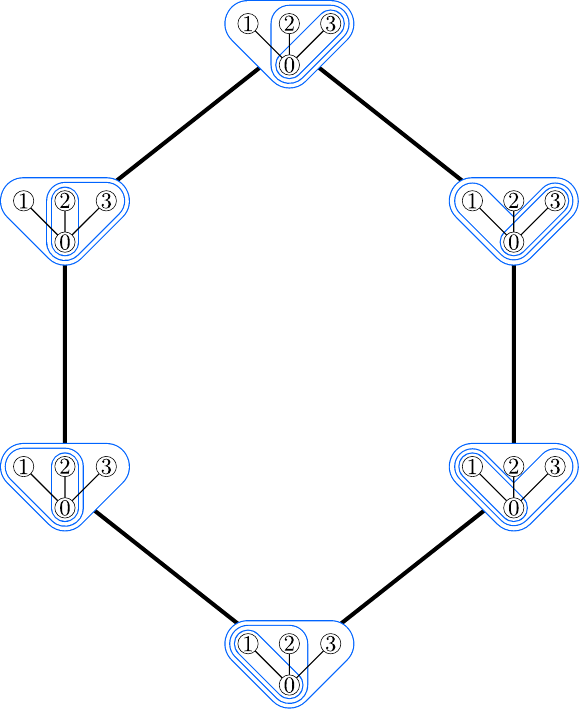}
  \end{center}
  \caption{On the left is the operahedron lattice of the chain in $\PT_3$, which is isomorphic to the third Tamari lattice. On the right is the operahedron lattice of the claw in $\PT_3$, which is isomorphic to the weak order on $\mathfrak S_3$. Each tube is circled in blue. }\label{fig:Tamari_weak} 
\end{figure}

The \dfn{preorder traversal} of a rooted plane tree $\TT$ is the ordering of the vertices of $\TT$ obtained by reading the root first and then reading the subtrees of the root, each in preorder, from left to right. For example, every tree appearing in \cref{fig:chain_claw_broom,fig:Tamari_weak,fig:operahedron1} has its vertex set identified with $\{0,1,\ldots,n\}$ (for the appropriate $n$) so that the preorder traversal is $0,1,\ldots,n$. 

Let $\TT\in\PT_n$, and let us identify the vertex set of $\TT$ with $\{0,1,\ldots,n\}$ so that $0,1,\ldots,n$ is the preorder traversal of $\TT$. A \dfn{maximal nesting} of $\TT$ is a nesting of $\TT$ that has $n$ tubes; in other words, it is a maximal element of ${\bf N}(\TT)$. Maximal nestings correspond to the vertices of the operahedron of $\TT$. Let $\MN(\TT)$ denote the set of maximal nestings of $\TT$. Say two maximal nestings of $\TT$ are \dfn{adjacent} if they correspond to vertices that are adjacent in the 1-skeleton of the operahedron of $\TT$. Suppose $\NN$ and $\NN'$ are adjacent maximal nestings of $\TT$. Then there exist $\tau\in\NN\setminus\NN'$ and $\tau'\in\NN'\setminus\NN$ such that $\NN\setminus\{\tau\}=\NN'\setminus\{\tau'\}$. Moreover, the following are equivalent:
\begin{enumerate}[(i)]
\item Every element of $\tau\setminus\tau'$ is less than every element of $\tau'\setminus\tau$ in $\mathbb Z$.
\item Some element of $\tau\setminus\tau'$ is less than some element of $\tau'\setminus\tau$ in $\mathbb Z$.
\end{enumerate} 
If these equivalent conditions hold, then we write $\NN\lessdot \NN'$. This defines the cover relations of a partial order $\leq$ on $\MN(\TT)$, which allows us to view $\MN(\TT)$ as a poset. If $\TT$ is a chain, then $\MN(\TT)$ is isomorphic to the $n$-th Tamari lattice; if $\TT$ is a claw, then $\MN(\TT)$ is isomorphic to the weak order on the symmetric group $\mathfrak S_n$ (see \cref{fig:Tamari_weak,fig:operahedron1}). 

\begin{figure}
\begin{center}\includegraphics[width=\linewidth]{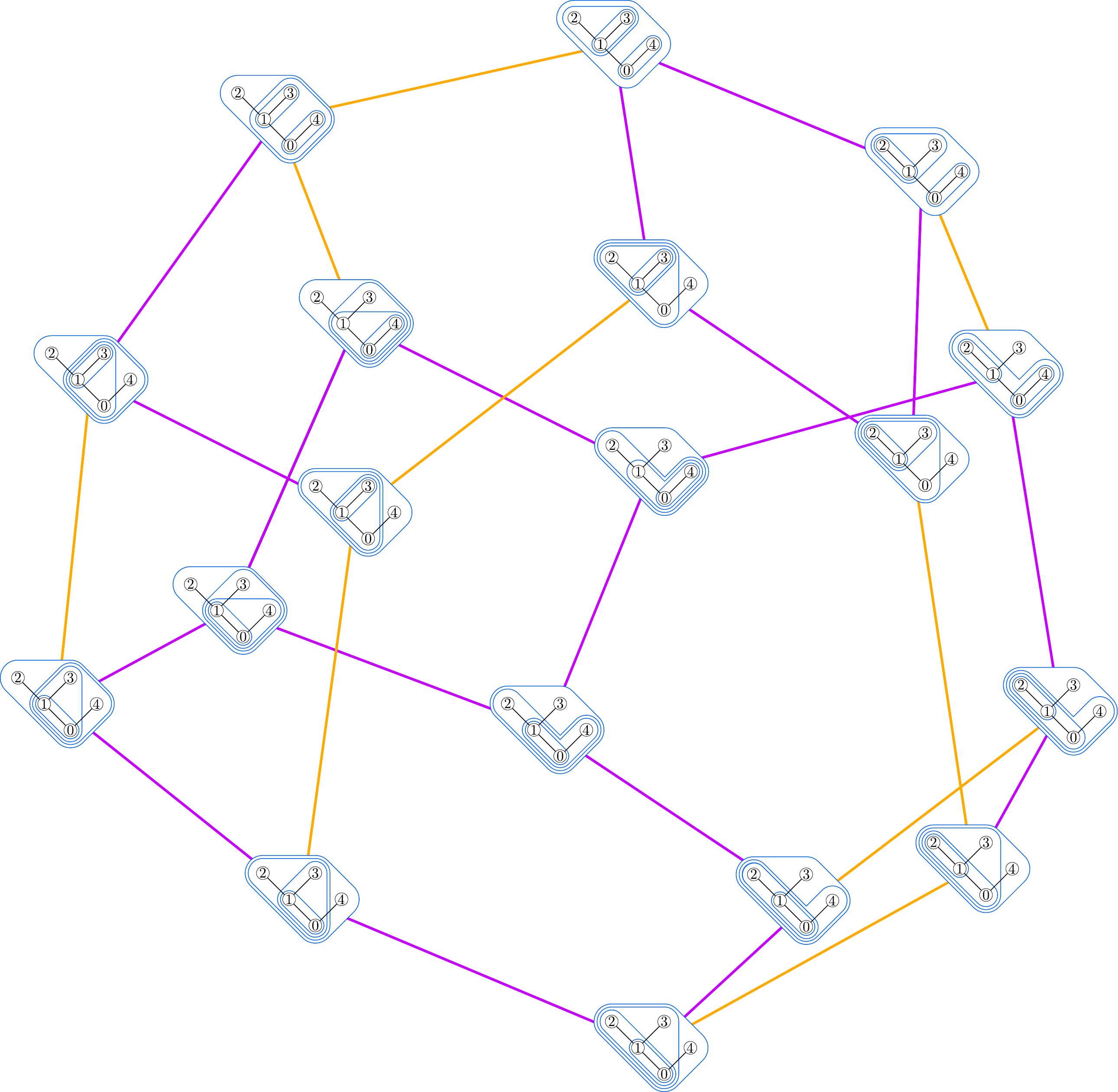}
  \end{center}
  \caption{The operahedron lattice of the tree $\begin{array}{l}\includegraphics[height=0.5cm]{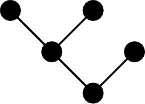}\end{array}$. We have identified the vertex set of the tree with the set $\{0,1,2,3,4\}$ so that $0,1,2,3,4$ is the preorder traversal. Each tube is circled in blue. Edges of the lattice corresponding to permutohedron moves are purple, while edges corresponding to associahedron moves are orange.}\label{fig:operahedron1}
\end{figure}

Laplante-Anfossi introduced the posets $\MN(\TT)$ and posed the problem of determining whether they are always lattices. Our first main result answers this question in the affirmative. 

\begin{theorem}\label{thm:lattice}
For every rooted plane tree $\TT$, the poset $\MN(\TT)$ is a lattice. 
\end{theorem}

In light of \cref{thm:lattice}, we call $\MN(\TT)$ an \dfn{operahedron lattice}. 

Let us say a rooted plane tree $\TT$ \dfn{contains} a rooted plane tree $\TT'$ if $\TT'$ can be obtained from $\TT$ by contracting edges. The following result is a useful tool for understanding the more refined structural properties of operahedron lattices.

\begin{proposition}\label{prop:interval}
Let $\TT$ and $\TT'$ be rooted plane trees. If $\TT$ contains $\TT'$, then $\MN(\TT')$ is isomorphic to an interval of $\MN(\TT)$. 
\end{proposition}

Using \cref{prop:interval} and the fact that intervals of distributive lattices are distributive, it is not difficult to check by hand that $\MN(\TT)$ is distributive if and only if $n\leq 2$. This characterization of distributivity is not too interesting, so we are naturally led to consider the family of \emph{semidistributive} lattices, which contains a more eclectic array of examples. Upon inspecting \cref{fig:operahedron1}, one can check directly that the operahedron lattice of the tree $\begin{array}{l}\includegraphics[height=0.5cm]{OperahedronPIC3}\end{array}$ is not semidistributive. Because intervals of semidistributive lattices are semidistributive, it follows from \cref{prop:interval} that the operahedron lattice of a tree that contains $\begin{array}{l}\includegraphics[height=0.5cm]{OperahedronPIC3}\end{array}$ cannot be semidistributive; we will prove that this is actually the only obstruction to semidistributivity. 

\begin{theorem}\label{thm:semidistributive}
Let $\TT$ be a rooted plane tree. The following are equivalent. 
\begin{enumerate}[label={\upshape(\Roman*)}, align=left, widest=iii, leftmargin=*]
\item\label{S1} The operahedron lattice $\MN(\TT)$ is semidistributive. 
\item\label{S2} The operahedron lattice $\MN(\TT)$ is meet-semidistributive. 
\item\label{S3} The operahedron lattice $\MN(\TT)$ is join-semidistributive. 
\item\label{S4} The tree $\TT$ does not contain the tree $\begin{array}{l}\includegraphics[height=0.5cm]{OperahedronPIC3}\end{array}$. 
\item\label{S5} Every vertex of $\TT$ that is not in the rightmost branch of $\TT$ is covered by at most $1$ element of~$\TT$. 
\end{enumerate}
\end{theorem}

When $\MN(\TT)$ is semidistributive, our proof of \cref{thm:semidistributive} provides a description of its join-irreducible elements and its meet-irreducible elements (see \cref{rem:explicit_join-irr}).  

Another generalization of the family of distributive lattices is the family of \emph{trim} lattices, which was introduced by Thomas \cite{Thomas} (see also \cites{Ungarian,Semidistrim,RiSM} for several notable examples and remarkable properties of trim lattices). Our next result characterizes the rooted plane trees whose operahedron lattices are trim. 

\begin{theorem}\label{thm:trim}
Let $\TT$ be a rooted plane tree. The operahedron lattice $\MN(\TT)$ is trim if and only if the root of $\TT$ is covered by at most $2$ elements of $\TT$ and every non-root vertex in $\TT$ is covered by at most $1$ element of $\TT$. 
\end{theorem}

\begin{remark}\label{rem:semidistrim}
Defant and Williams \cite{Semidistrim} introduced the family of \emph{semidistrim} lattices and proved that semidistributive lattices and trim lattices are semidistrim. We have checked that the operahedron lattice of the tree $\begin{array}{l}\includegraphics[height=0.5cm]{OperahedronPIC3}\end{array}$ (shown in \cref{fig:operahedron1}) is not semidistrim. Since intervals of semidistrim lattices are semidistrim \cite{Semidistrim}*{Theorem~7.8}, it follows that the operahedron lattice $\MN(\TT)$ of a tree $\TT$ is semidistrim if and only if the five equivalent conditions in \cref{thm:semidistributive} hold. 
\end{remark}

We denote by $\mathfrak S_n$ the $n$-th symmetric group, which consists of the permutations of the set ${[n]=\{1,\ldots,n\}}$. Let $\s\colon\mathfrak S_n\to\mathfrak S_n$ denote West's \emph{stack-sorting map} (see \cref{sec:stacks} for the definition of this map). West \cite{West} introduced this function as a deterministic analogue of Knuth's stack-sorting machine \cite{Knuth}. It has now been studied vigorously in combinatorics and computer science \cites{Bona_survey_2019,Branden,DefantCatalan,DefantCounting,DefantThesis,DefantEngenMiller} and has found striking connections with free probability theory \cite{DefantTroupes} and polyhedral geometry \cites{DefantFertilitopes,SackNguyen,LeeMitchellVindas}. 

Let $\Weak(\mathfrak S_n)$ denote the (right) weak order on $\mathfrak S_n$. For $w\in\mathfrak S_n$, let $\Delta_{\Weak(\mathfrak S_n)}(w)$ be the order ideal of $\Weak(\mathfrak S_n)$ generated by $w$, viewed as a subposet of $\Weak(\mathfrak S_n)$. Let \[\wo(k,n)=k(k-1)\cdots 1(k+1)(k+2)\cdots n\in \mathfrak S_n.\] Note that $\Delta_{\Weak(\mathfrak S_n)}(\wo(k,n))=\{u\in \mathfrak S_n:u(i)=i\text{ for all }k+1\leq i\leq n\}$. 

The final direction that we will explore connects West's stack-sorting map with operahedron lattices of brooms. This line of investigation was initiated by Nguyen and Sack \cite{SackNguyen}, who found that the $h$-vector of the operahedron of $\Broom_{k,n}$ counts permutations in $\s^{-1}(\Delta_{\Weak(\mathfrak S_n)}(\wo(k,n)))$ according to the descent statistic.  

\begin{theorem}\label{thm:stack-sorting}
Fix positive integers $k\leq n$. The operahedron lattice $\MN(\Broom_{k,n})$ is isomorphic to the subposet $\s^{-1}(\Delta_{\Weak(\mathfrak S_n)}(\wo(k,n)))$ of $\Weak(\mathfrak S_n)$. 
\end{theorem}

Our proof of \cref{thm:stack-sorting} constructs an explicit isomorphism. 

It is not obvious \emph{a priori} that the subposet $\s^{-1}(\Delta_{\Weak(\mathfrak S_n)}(\wo(k,n)))$ of $\Weak(\mathfrak S_n)$ is a lattice, yet combining \cref{thm:semidistributive,thm:stack-sorting} yields the following even stronger corollary. 

\begin{corollary}
Fix positive integers $k\leq n$. The subposet $\s^{-1}(\Delta_{\Weak(\mathfrak S_n)}(\wo(k,n)))$ of $\Weak(\mathfrak S_n)$ is a semidistributive lattice. 
\end{corollary}

The remainder of the paper is organized as follows. In \cref{sec:preliminaries}, we collect necessary notation and terminology pertaining to posets, lattices, permutations, and rooted plane trees. In \cref{sec:lattice}, we prove that $\MN(\TT)$ is isomorphic to a different poset $\Theta(\TT^\times)$, and we prove that this latter poset is a lattice, thereby establishing \cref{thm:lattice}. \cref{sec:intervals,sec:semidistributivity,sec:trim} are devoted to proving \cref{prop:interval}, \cref{thm:semidistributive}, and \cref{thm:trim}, respectively. In \cref{sec:stacks}, we discuss the stack-sorting map and prove \cref{thm:stack-sorting}. We conclude the paper with suggestions for future work in \cref{sec:conclusion}. 

\section{Preliminaries}\label{sec:preliminaries}
 
We assume basic familiarity with the theory of posets (partially ordered sets); a standard reference for this topic is \cite{Stanley}*{Chapter~3}. All posets in this article are assumed to be finite. 

Let $P$ be a poset with partial order $\leq$. For $x,y\in P$ with $x\leq y$, the \dfn{interval} from $x$ to $y$ is the set $[x,y]=\{z\in P:x\leq z\leq y\}$, which we view as a subposet of $P$. If $x<y$ and $[x,y]=\{x,y\}$, then we say $y$ \dfn{covers} $x$ and write $x\lessdot y$. For $x\in P$, let \[\Delta_P(x)=\{z\in P:z\leq x\}\quad\text{and}\quad\nabla_P(x)=\{z\in P:x\leq z\}.\] A \dfn{chain} of $P$ is a totally ordered subset of $P$. We often represent a chain of $P$ as a sequence ${x_0<x_1<\cdots< x_\ell}$; the \dfn{length} of this chain is the number $\ell$. The \dfn{height} of $P$, which we denote by $\mathrm{height}(P)$, is the maximum length of a chain of $P$. 

Suppose $P_1$ and $P_2$ are posets with partial orders $\leq_1$ and $\leq_2$, respectively. A map $\varphi\colon P_1\to P_2$ is said to be \dfn{order-preserving} if $\varphi(x)\leq_2\varphi(y)$ for all $x,y\in P_1$ satisfying $x\leq_1 y$. The \dfn{product} of $P_1$ and $P_2$ is the poset $P_1\times P_2$ whose underlying set is the cartesian product of $P_1$ and $P_2$ and whose order relation $\leq$ is such that $(x_1,x_2)\leq (y_1,y_2)$ if and only if $x_1\leq_1 y_1$ and $x_2\leq_2 y_2$. 

Let $P$ be a poset with $N$ elements. A \dfn{linear extension} of $P$ is a word $\sigma=\sigma(1)\cdots\sigma(N)$ such that ${P=\{\sigma(1),\ldots,\sigma(N)\}}$ and such that $\sigma(i)\leq\sigma(j)$ whenever $i\leq j$ in $\mathbb{Z}$. Let $\LL(P)$ denote the set of linear extensions of $P$. For $\sigma\in\LL(P)$, we consider the strict total order $\prec_\sigma$ on $P$ defined so that $x\prec_\sigma y$ if and only if $x$ precedes $y$ in $\sigma$. We also write $x\preceq_\sigma y$ to mean $x\prec_\sigma y$ or $x=y$. It will be helpful to extend this notation to subsets of $P$ as well. For $X,Y\subseteq P$, we write $X\prec_\sigma Y$ if $x\prec_\sigma y$ for all $x\in X$ and $y\in Y$. (Note that if $X\prec_\sigma Y$, then $X\cap Y=\emptyset$.) A \dfn{consecutive factor} of $\sigma$ is a word of the form $\sigma(a)\sigma(a+1)\cdots\sigma(b)$ for $1\leq a\leq b\leq N$. Given a set $X$, we write $\sigma\vert_X$ for the word obtained from $\sigma$ by deleting the entries from $P\setminus X$.   

We say a poset $L$ is a \dfn{lattice} if any two elements $x,y\in L$ have a greatest lower bound, which is called their \dfn{meet} and denoted by $x\wedge y$, and a least upper bound, which is called their \dfn{join} and denoted by $x\vee y$. Let $L$ be a lattice. An element of $L$ is \dfn{join-irreducible} if it covers exactly $1$ element of $L$. Dually, an element of $L$ is \dfn{meet-irreducible} if it is covered by exactly $1$ element of $L$. Let $\mathcal J_L$ and $\mathcal M_L$ be the set of join-irreducible elements of $L$ and the set of meet-irreducible elements of $L$, respectively. It is a basic fact that $\mathrm{height}(L)\leq|\mathcal J_L|$ and $\mathrm{height}(L)\leq|\mathcal M_L|$. We say $L$ is \dfn{extremal} if $\mathrm{height}(L)=|\mathcal J_L|=|\mathcal M_L|$. An element $u\in L$ is \dfn{left-modular} if for all $v,w\in L$ satisfying $v<w$, we have $(v\vee u)\wedge w=v\vee (u\wedge w)$. The lattice $L$ is called \dfn{left-modular} if it has a maximal chain whose elements are all left-modular. We say $L$ is \dfn{trim} if it is both extremal and left-modular. We say $L$ is \dfn{meet-semidistributive} if for all elements $x,y,z\in L$ satisfying $x\wedge y=x\wedge z$, we have $x\wedge y=x\wedge(y\vee z)$. We say $L$ is \dfn{join-semidistributive} if for all $x,y,z\in L$ satisfying $x\vee y=x\vee z$, we have $x\vee y=x\vee(y\wedge z)$. We say $L$ is \dfn{semidistributive} if it is both meet-semidistributive and join-semidistributive. 

If $L_1$ and $L_2$ are lattices, then their product $L_1\times L_2$ is a lattice, and 
\begin{equation}\label{eq:product_join-irr}
\mathcal J_{L_1\times L_2}=(\mathcal J_{L_1}\times\{\hat 0_2\})\sqcup(\{\hat 0_1\}\times \mathcal J_{L_2}),
\end{equation}
where $\hat 0_1$ and $\hat 0_2$ are the unique minimal elements of $L_1$ and $L_2$, respectively. 

We represent a permutation $w$ in the symmetric group $\mathfrak S_n$ via its \dfn{one-line notation} $w(1)\cdots w(n)$. An \dfn{inversion} of $w$ is a pair $(i,j)$ such that $1\leq i<j\leq n$ and $w^{-1}(j)<w^{-1}(i)$. Let $\Inv(w)$ denote the set of inversions of $w$. The \dfn{(right) weak order} on $\mathfrak S_n$ is the poset $\Weak(\mathfrak S_n)=(\mathfrak S_n,\leq)$, where for $w,w'\in\mathfrak S_n$, we have $w\leq w'$ if and only if $\Inv(w)\subseteq\Inv(w')$. It is known \cite{BjornerBrenti} that $\Weak(\mathfrak S_n)$ is a lattice. 

A \dfn{rooted plane tree} is a rooted tree in which the subtrees of each vertex are linearly ordered from left to right. As before, we let $\PT_n$ denote the set of rooted plane trees with $n+1$ vertices. We draw a rooted plane tree $\TT$ with vertex set $\VV$ as the Hasse diagram of the poset $(\VV,\leq_\TT)$ (so the root is at the bottom). Recall that for $v\in\VV^\times$, we write $\nabla_\TT(v)=\{v'\in\VV^\times:v\leq_\TT v'\}$. 

Let $\TT$ be a rooted plane tree with vertex set $\VV$ and root vertex ${\bf r}$. We write $\TT^\times$ for the poset obtained from $\TT$ by deleting ${\bf r}$. We can also view $\TT^\times$ as a forest graph (via its Hasse diagram). We let $\VV^\times=\VV\setminus\{{\bf r}\}$ denote the vertex set of $\TT^\times$. Define the \dfn{rightmost branch} of $\TT$ as follows. If $\TT$ has only $1$ vertex, then the rightmost branch of $\TT$ is $\{{\bf r}\}$. Now suppose $\TT$ has at least $2$ vertices, and let ${\bf r}'$ be the rightmost vertex that covers ${\bf r}$. Let $\TT_{{\bf r}'}$ be the subtree of $\TT$ with vertex set $\nabla_{\TT}({\bf r}')$. Then the rightmost branch of $\TT$ is $\{{\bf r}\}\cup\mathsf{B}_{{\bf r}'}$, where $\mathsf{B}_{{\bf r}'}$ is the rightmost branch of $\TT_{{\bf r}'}$. 

\section{The Lattice Property}\label{sec:lattice}

Let $\TT$ be a rooted plane tree with vertex set $\VV$ and root vertex ${\bf r}$. As above, let $\TT^\times$ be the forest poset with vertex set $\VV^\times=\VV\setminus\{{\bf r}\}$. 

It will be useful to distinguish two different types of cover relations in $\MN(\TT)$. Suppose $\NN\lessdot\NN'$ is a cover relation in $\MN(\TT)$, and let $\tau\in\NN\setminus\NN'$ and $\tau'\in\NN'\setminus\NN$ be such that $\NN\setminus\{\tau\}=\NN'\setminus\{\tau'\}$. If $\tau$ and $\tau'$ have the same root, then we say $\NN$ and $\NN'$ are related by a \dfn{permutohedron move}. If $\tau$ and $\tau'$ have different roots, then we say $\NN$ and $\NN'$ are related by an \dfn{associahedron move}. 

An \dfn{ornament} of $\TT^\times$ is a subset of $\VV^\times$ that induces a connected subgraph of $\TT^\times$. Every ornament $\mathfrak o$ has a unique minimal element $v_\mathfrak{o}$; we say $\mathfrak o$ is \dfn{hung} at $v_{\mathfrak{o}}$. Let $\Orn(\TT^\times)$ denote the set of ornaments of $\TT^\times$. An \dfn{ornamentation} of $\TT^\times$ is a function $\varrho\colon\VV^\times\to\Orn(\TT^\times)$ such that \begin{itemize}
\item for all $v\in \VV^\times$, the ornament $\varrho(v)$ is hung at $v$; 
\item for all $v,v'\in\VV^\times$, the ornaments $\varrho(v)$ and $\varrho(v')$ are either nested or disjoint.
\end{itemize} 
An ornament $\varrho(v)$ of an ornamentation $\varrho$ is \dfn{maximal} if there does not exist $v'\in\VV^\times\setminus\{v\}$ such that $\varrho(v)\subseteq\varrho(v')$.
Let $\OO(\TT^\times)$ denote the set of ornamentations of $\TT^\times$. There is a natural partial order $\leq$ on $\OO(\TT^\times)$ defined so that $\varrho\leq\varrho'$ if and only if $\varrho(v)\subseteq\varrho'(v)$ for all $v\in\VV^\times$. 

The poset $\OO(\TT^\times)$ has a unique minimal element $\varrho_{\min}$ and a unique maximal element $\varrho_{\max}$; they are defined so that \[\varrho_{\min}(v)=\{v\}\quad\text{and}\quad\varrho_{\max}(v)=\nabla_\TT(v)\] for all $v\in\VV^\times$.

\begin{proposition}\label{prop:ornamentation_lattice}
Let $\TT$ be a rooted plane tree. The poset $\OO(\TT^\times)$ is a lattice. 
\end{proposition}
\begin{proof}
The meet operation on $\OO(\TT^\times)$ is given by \[(\varrho\wedge\varrho')(v)=\varrho(v)\cap\varrho'(v)\] for all $v,v'\in\VV^\times$. The join $\varrho\vee\varrho'$ is simply the meet of the set of upper bounds of $\rho$ and $\rho'$ (this set is nonempty because $\OO(\TT^\times)$ has a unique maximal element $\varrho_{\max}$). 
\end{proof} 

In light of \cref{prop:ornamentation_lattice}, we call $\OO(\TT^\times)$ an \dfn{ornamentation lattice}. 

Let us identify $\VV$ with $\{0,1,\ldots,n\}$ so that $0,1,\ldots,n$ is the preorder traversal of $\TT$. Then the vertex set of $\TT^\times$ is $[n]$, so the set $\LL(\TT^\times)$ of linear extensions of $\TT^\times$ can be viewed as a subset of $\mathfrak S_n$ (a linear extension is just the one-line notation of a permutation). It follows from \cite{BjornerWachs}*{Theorem~6.8} that $\LL(\TT^\times)$ is an interval in $\Weak(\mathfrak S_n)$. In particular, $\LL(\TT^\times)$ is a lattice. 

In order to prove that $\MN(\TT)$ is a lattice, we will first embed it as a subposet of the product lattice $\LL(\TT^\times)\times \OO(\TT^\times)$. Suppose $\NN\in\MN(\TT)$ is a maximal nesting of $\TT$. For each $v\in\VV^\times$, let $\varrho_{\NN}(v)$ be the largest tube in $\NN$ whose root is $v$; if no such tube exists, let $\varrho_{\NN}(v)=\{v\}$. This defines an ornamentation $\varrho_\NN\in\OO(\TT^\times)$. There is a unique linear extension $\widetilde\lambda_\NN\in\LL(\TT)$ such that for every $\tau\in\NN$, the elements of $\tau$ form a consecutive factor of $\widetilde\lambda_\NN$. The first entry in the word $\widetilde\lambda_\NN$ is the root vertex ${\bf r}$. Let $\lambda_\NN\in\LL(\TT^\times)$ be the linear extension of $\TT^\times$ obtained from $\widetilde\lambda_\NN$ by deleting ${\bf r}$.

\begin{figure}[ht]
  \begin{center}
  \includegraphics[width=\linewidth]{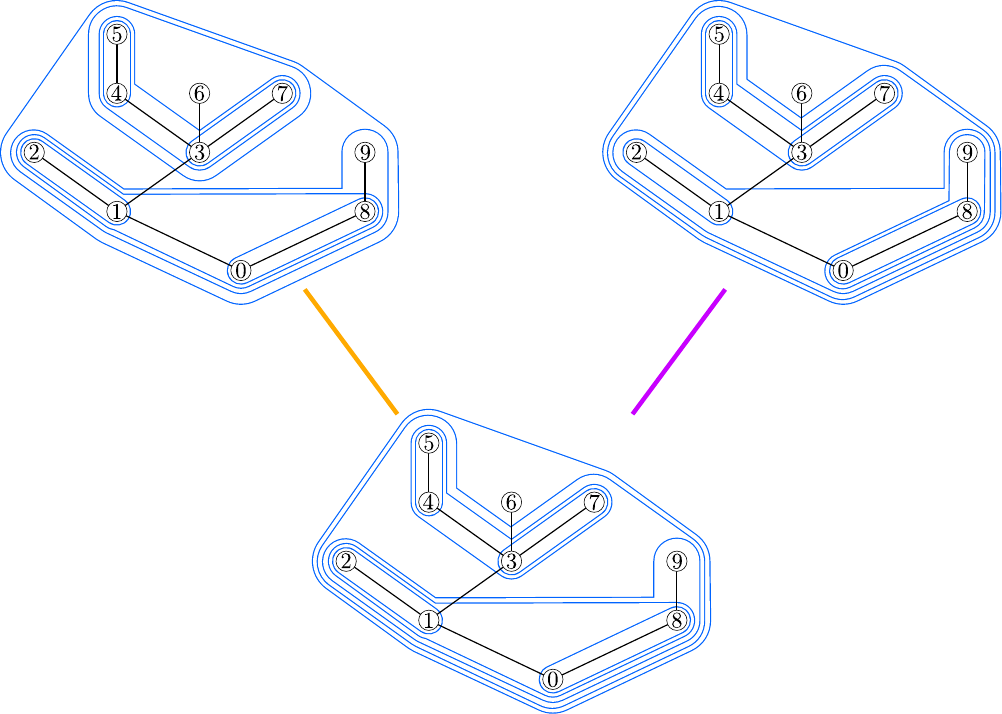}
  \end{center}
  \caption{Three maximal nestings of a tree in $\PT_9$. The top two maximal nestings cover the bottom maximal nesting. The cover relation on the left corresponds to an associahedron move, while the cover relation on the right corresponds to a permutohedron move.}\label{fig:Psi1}
\end{figure}

\begin{figure}[ht]
  \begin{center}
  \includegraphics[width=\linewidth]{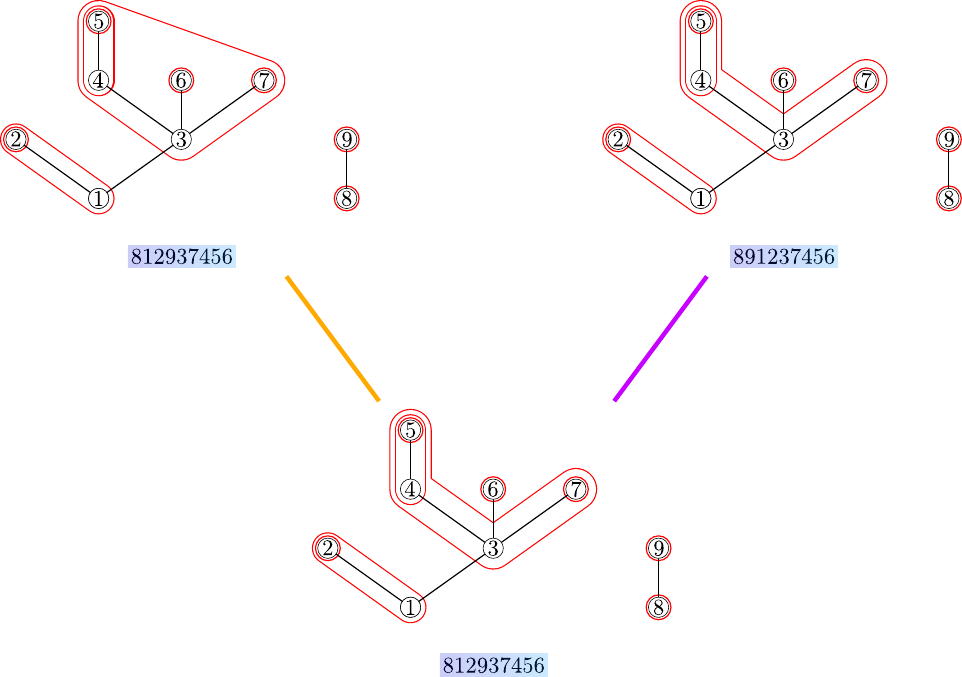}
  \end{center}
  \caption{Applying $\Psi$ to the maximal nestings in \cref{fig:Psi1} yields these three pairs, each of which consists of an ornamentation (which we represent by circling the ornaments in red) and a linear extension.}\label{fig:Psi2}
\end{figure}

\begin{remark}\label{rem:covers}
Suppose $\NN\lessdot\NN'$ is a cover relation in $\MN(\TT)$. Let $\tau\in\NN\setminus\NN'$ and $\tau'\in\NN'\setminus\NN$ be such that $\NN\setminus\{\tau\}=\NN'\setminus\{\tau'\}$. Let $v$ and $v'$ be the roots of $\tau$ and $\tau'$, respectively. If $\NN$ and $\NN'$ are related by a permutohedron move (meaning $v=v'$), then \begin{itemize}
\item $\varrho_\NN=\varrho_{\NN'}$; \item $\lambda_\NN\vert_{\tau\cap\tau'}\lambda_\NN\vert_{\tau\setminus\tau'}\lambda_\NN\vert_{\tau'\setminus\tau}$ is a consecutive factor of $\lambda_\NN$; 
\item $\lambda_{\NN'}$ is obtained from $\lambda_\NN$ by swapping $\lambda_\NN\vert_{\tau\setminus\tau'}$ and $\lambda_\NN\vert_{\tau'\setminus \tau}$. 
\end{itemize}
On the other hand, if $\NN$ and $\NN'$ are related by an associahedron move (meaning $v\neq v'$), then \begin{itemize}
\item $\varrho_\NN(u)=\varrho_{\NN'}(u)$ for all $u\in\VV^\times\setminus\{v'\}$; 
\item $\varrho_\NN(v')=\tau\cap\tau'$;
\item $\varrho_{\NN'}(v')=\tau'$;
\item $\lambda_\NN$ and $\lambda_{\NN'}$ are equal and have $\lambda_\NN\vert_{\tau\setminus\tau'}\lambda_\NN\vert_{\tau\cap\tau'}\lambda_\NN\vert_{\tau'\setminus\tau}$ as a consecutive factor. 
\end{itemize}
See \cref{fig:Psi1,fig:Psi2}. 
\end{remark}

Let $\Theta(\TT^\times)$ be the set of pairs $(\lambda,\varrho)\in\LL(\TT^\times)\times\OO(\TT^\times)$ such that for every $v\in\VV^\times$, the elements of $\varrho(v)$ form a consecutive factor of $\lambda$. We will view $\Theta(\TT^\times)$ as a subposet of $\LL(\TT^\times)\times \OO(\TT^\times)$. Note that $(\lambda_\NN,\varrho_\NN)\in\Theta(\TT^\times)$ for all $\NN\in\MN(\TT)$. Thus, we can define a map $\Psi\colon\MN(\TT)\to\Theta(\TT^\times)$ by 
\[\Psi(\NN)=(\lambda_\NN,\varrho_\NN).\] For example, if we apply $\Psi$ to the maximal nestings in \cref{fig:Psi1}, we obtain the pairs shown in \cref{fig:Psi2}.  

\begin{lemma}\label{lem:bijection_Psi}
For $\TT\in\PT_n$, the map $\Psi\colon\MN(\TT)\to\Theta(\TT^\times)$ is a bijection. 
\end{lemma} 

\begin{proof}
Let $(\lambda,\varrho)\in\Theta(\TT^\times)$. We will show that there is a unique $\NN\in\MN(\TT)$ such that $\lambda_\NN=\lambda$ and $\varrho_\NN=\varrho$. This is trivial if $n\leq 1$, so we may assume $n\geq 2$ and proceed by induction on~$n$. 

Let $M$ be the set of vertices $v\in\VV^\times$ such that $\varrho(v)$ is a maximal ornament of $\varrho$. Let $v_1,\ldots,v_m$ be the elements of $M$, listed in the order they appear in $\lambda$. For $i\in[m]$, let $\TT_i$ be the subtree of $\TT$ with vertex set $\varrho(v_i)$.

Fix $i\in[m]$. Let $\TT_i^\times$ be the forest poset obtained from $\TT_i$ by deleting its root vertex $v_i$. Let ${\VV_i^\times=\varrho(v_i)\setminus\{v_i\}}$ be the vertex set of $\TT_i^\times$. Consider the ornamentation $\varrho_{i}\in\OO(\TT_i^\times)$ obtained by restricting $\varrho$ to $\VV_i^\times$. Let $\lambda_i=\lambda\vert_{\VV_i^\times}$. Note that $(\lambda_i,\varrho_i)\in\Theta(\TT_i^\times)$. By induction, we know that there is a unique $\NN_i\in\MN(\TT_i)$ such that $\lambda_{\NN_i}=\lambda_i$ and $\varrho_{\NN_i}=\varrho_i$. For each $1\leq j\leq m$, it follows from the definition of $\Theta(\TT^\times)$ that the set $\tau_j=\{0\}\cup\varrho(v_1)\cup\cdots\cup\varrho(v_j)$ is a tube of $\TT$. Let \[\NN=\{\tau_1,\ldots,\tau_m\}\cup\NN_1\cup\cdots\cup\NN_m.\] Then $\Psi(\NN)=(\lambda,\varrho)$, and $\NN$ is the unique maximal nesting of $\TT$ with this property. 
\end{proof}

Consider a cover relation $(\lambda,\varrho)\lessdot(\lambda',\varrho')$ in $\Theta(\TT^\times)$. We say $(\lambda,\varrho)$ and $(\lambda',\varrho')$ are related by a \dfn{permutohedron move} if $\varrho=\varrho'$ and there exist vertices $p,q\in\VV^\times$ such that 
\begin{itemize}
\item $p$ and $q$ are incomparable in $\TT$; 
\item every number in $\varrho(p)$ is less than every number in $\varrho(q)$ in $\mathbb Z$; 
\item $\lambda\vert_{\varrho(p)}\lambda\vert_{\varrho(q)}$ is a consecutive factor of $\lambda$; 
\item $\lambda'$ is obtained from $\lambda$ by swapping $\lambda\vert_{\varrho(p)}$ and $\lambda\vert_{\varrho(q)}$. 
\end{itemize}
For example, if $(\lambda,\varrho)$ and $(\lambda',\varrho')$ are the pairs on the bottom and the top right of \cref{fig:Psi2}, then $(\lambda,\varrho)$ and $(\lambda',\varrho')$ are related by a permutohedron move. In this example, the linear extension $\lambda'=891237456$ is obtained from the linear extension $\lambda=812937456$ by swapping the consecutive factors $\lambda\vert_{\varrho(1)}=12$ and $\lambda\vert_{\varrho(9)}=9$. On the other hand, we say $(\lambda,\varrho)$ and $(\lambda',\varrho')$ are related by an \dfn{associahedron move} if $\lambda=\lambda'$ and there exists $t\in\VV^\times$ such that $\varrho(v)=\varrho'(v)$ for all $v\in\VV^\times\setminus\{t\}$ and $\varrho'(t)=\varrho(t)\cup\varrho(t^{\to})$, where $t^\to$ is the vertex in $\VV^\times$ appearing immediately after $\lambda\vert_{\varrho(v)}$ in $\lambda$. For example, if $(\lambda,\varrho)$ and $(\lambda',\varrho')$ are the pairs on the bottom and the top left of \cref{fig:Psi2}, then $(\lambda,\varrho)$ and $(\lambda',\varrho')$ are related by an associahedron move. In this example, we have $\varrho'(3)=\varrho(3)\cup\varrho(6)$ because $6$ is the vertex appearing immediately after the consecutive factor $\lambda\vert_{\varrho(3)}=3745$ in $\lambda$.  

\begin{lemma}\label{lem:Theta_covers}
Let $(\lambda,\varrho)\lessdot(\lambda',\varrho')$ be a cover relation in $\Theta(\TT^\times)$. Either $(\lambda,\varrho)$ and $(\lambda',\varrho')$ are related by a permutohedron move, or they are related by an associahedron move.  
\end{lemma}

\begin{proof}
It suffices to show that $\varrho=\varrho'$ or $\lambda=\lambda'$. Indeed, if $\varrho=\varrho'$, then it is straightforward to show that $(\lambda,\varrho)$ and $(\lambda',\varrho')$ are related by a permutohedron move, and if $\lambda=\lambda'$, then it is straightforward to show that $(\lambda,\varrho)$ and $(\lambda',\varrho')$ are related by an associahedron move. Assume by way of contradiction that $\varrho<\varrho'$ and $\lambda<\lambda'$. 

Let $M$ be the set of vertices $v\in\VV^\times$ such that $\varrho(v)$ is a maximal ornament of $\varrho$. Let $v_1,\ldots,v_m$ be the elements of $M$, listed in the order they appear in $\lambda$. Then $\lambda=\lambda\vert_{\varrho(v_1)}\cdots\lambda\vert_{\varrho(v_m)}$. Suppose by way of contradiction that there is an index $i\in\{2,\ldots,m\}$ such that $\varrho(v_{i})\prec_{\lambda'}\varrho(v_{i-1})$. Then there is a linear extension $\lambda''\in\LL(\TT^\times)$ obtained from $\lambda$ by swapping the factors $\lambda\vert_{\varrho(v_{i-1})}$ and $\lambda\vert_{\varrho(v_i)}$. We have $(\lambda'',\varrho)\in\Theta(\TT^\times)$ and $(\lambda,\varrho)<(\lambda'',\varrho)<(\lambda',\varrho')$, which contradicts the assumption that $(\lambda,\varrho)$ is covered by $(\lambda',\varrho')$ in $\Theta(\TT^\times)$. From this contradiction, we deduce that for each $i\in\{2,\ldots,m\}$, there exist $\ell_{i-1}\in\varrho(v_{i-1})$ and $r_i\in\varrho(v_i)$ such that $\ell_{i-1}\prec_{\lambda'} r_i$. 

Assume for the moment that $\varrho(v_1),\ldots,\varrho(v_m)$ are also the maximal ornaments of $\varrho'$. It follows from the preceding paragraph that the elements of $M$ appear in the order $v_1,\ldots,v_m$ in $\lambda'$. For $i\in[m]$, let $\TT_i$ be the subtree of $\TT$ with vertex set $\varrho(v_i)$. Let $\VV_i^\times=\varrho(v_i)\setminus\{v_i\}$ be the vertex set of the forest $\TT_i^\times$. Let $\lambda_i=\lambda\vert_{\VV_i^\times}$ and $\lambda_i'=\lambda'\vert_{\VV_i^\times}$, and note that $\lambda_i,\lambda_i'\in\LL(\TT_i^\times)$. Let $\varrho_i$ and $\varrho_i'$ be the ornamentations in $\OO(\TT^\times)$ obtained by restricting $\varrho$ and $\varrho'$, respectively, to $\VV_i^\times$. If there exists $i\in[m]$ such that $\varrho_i<\varrho_i'$ and $\lambda_i<\lambda_i'$, then we can use induction to find that $(\lambda_i,\varrho_i)$ is not covered by $(\lambda_i',\varrho_i')$ in $\Theta(\TT_i^\times)$; this then contradicts the assumption that $(\lambda,\varrho)$ is covered by $(\lambda',\varrho')$ in $\Theta(\TT^\times)$. Otherwise, there exist distinct indices $i,i'\in[m]$ such that $\lambda_i<\lambda_i'$, $\varrho_i=\varrho_i'$, $\lambda_{i'}=\lambda_{i'}'$, and $\varrho_{i'}<\varrho_{i'}'$. In this case, we can consider the linear extension $\lambda'''\in\LL(\TT^\times)$ obtained from $\lambda$ by reordering the elements of $\varrho(v_i)$ into the order they appear in $\lambda'$; we have $(\lambda''',\varrho)\in\Theta(\TT^\times)$ and $(\lambda,\varrho)<(\lambda''',\varrho)<(\lambda',\varrho')$, which is again a contradiction.

We may now assume that $\{\varrho(v_1),\ldots,\varrho(v_m)\}$ is not the set of maximal ornaments of $\varrho'$. Thus, there exists $j\in[m]$ such that $\varrho'(v_j)$ is a maximal ornament of $\varrho'$ and such that $\varrho(v_j)\subsetneq\varrho'(v_j)$. Let $k$ be the smallest index in $[m]\setminus\{j\}$ such that $\varrho(v_k)\subseteq\varrho'(v_j)$. Because $\lambda$ is a linear extension, every number in $\varrho'(v_j)\setminus\{v_j\}$ appears to the right of $v_j$ in $\lambda$. Therefore, $j<k$. 

Suppose $k=j+1$. Let $\varrho''$ be the ornamentation of $\TT^\times$ such that $\varrho''(v_j)=\varrho(v_j)\cup\varrho(v_k)$ and $\varrho''(v)=\varrho(v)$ for all $v\in\VV^\times\setminus\{v_j\}$. Then $(\lambda,\varrho'')\in\Theta(\TT^\times)$, and $\varrho<\varrho''\leq\varrho'$. This implies that $(\lambda,\varrho)<(\lambda,\varrho'')<(\lambda',\varrho')$, which is a contradiction. From this, we deduce that $k>j+1$. 

For each $i\in\{j+1,\ldots,k\}$, the set $\varrho(v_i)$ is contained in a maximal ornament of $\varrho'$, so the assumption that $\varrho'(v_j)$ is a maximal ornament of $\varrho'$ guarantees that either ${\varrho(v_i)\prec_{\lambda'}\{v_j\}}$ or ${\{v_j\}\prec_{\lambda'}\varrho(v_i)}$. If there exists $i\in\{j+1,\ldots,k\}$ such that $\varrho(v_i)\prec_{\lambda'}\{v_j\}$, then we can choose this index $i$ minimally to find that $r_i\prec_{\lambda'}\ell_{i-1}$, which is impossible. Therefore, $\{v_j\}\prec_{\lambda'}(\varrho(v_{j+1})\cup\cdots\cup\varrho(v_k))$. This implies that $v_j\prec_{\lambda'}\ell_{k-1}\prec_{\lambda'} r_k$. Because $v_j,r_k\in\varrho(v_j)\subseteq\varrho'(v_j)$ and the numbers in $\varrho'(v_j)$ form a consecutive factor of $\lambda'$, this implies that $\ell_{k-1}\in\varrho'(v_j)$. It follows that $\varrho(v_{k-1})\subseteq\varrho'(v_j)$, which contradicts the minimality in the definition of $k$. 
\end{proof}

It follows from \cref{rem:covers,lem:Theta_covers} that $\NN\lessdot\NN'$ is a cover relation in $\MN(\TT)$ corresponding to a permutohedron (respectively, associahedron) move if and only if $\Psi(\NN)\lessdot\Psi(\NN')$ is a cover relation in $\Theta(\TT^\times)$ corresponding to a permutohedron (respectively, associahedron) move. This proves the following proposition. 

\begin{proposition}\label{prop:isomorphism}
For $\TT\in\PT_n$, the map $\Psi\colon\MN(\TT)\to\Theta(\TT^\times)$ is a poset isomorphism. 
\end{proposition} 

The remainder of this section is devoted to proving that $\Theta(\TT^\times)$ is a lattice; our proof relies on the following result due to Bj\"orner, Edelman, and Ziegler. 

\begin{proposition}[{\cite{BEZ}}]\label{lem:BEZ}
Let $P$ be a finite poset with a unique minimal element and a unique maximal element. Suppose that for all distinct $x_0,x_1,x_2\in P$ satisfying $x_0\lessdot x_1$ and $x_0\lessdot x_2$, the elements $x_1$ and $x_2$ have a least upper bound in $P$. Then $P$ is a lattice. 
\end{proposition}

\begin{proof}[Proof of \cref{thm:lattice}]
By \cref{prop:isomorphism}, it suffices to prove that $\Theta(\TT^\times)$ is a lattice. Let $(\lambda_0,\varrho_0)$, $(\lambda_1,\varrho_1)$, $(\lambda_2,\varrho_2)$ be distinct elements of $\Theta(\TT^\times)$ such that $(\lambda_0,\varrho_0)\lessdot(\lambda_1,\varrho_1)$ and $(\lambda_0,\varrho_0)\lessdot (\lambda_2,\varrho_2)$. We will prove that $(\lambda_1,\varrho_1)$ and $(\lambda_2,\varrho_2)$ have a least upper bound in $\Theta(\TT^\times)$; in light of \cref{lem:BEZ}, this will complete the proof. We consider three cases based on whether the cover relations under consideration correspond to permutohedron or associahedron moves. By swapping the roles of $(\lambda_1,\varrho_1)$ and $(\lambda_2,\varrho_2)$ if necessary, we may assume that either the cover relation $(\lambda_0,\varrho_0)\lessdot (\lambda_1,\varrho_1)$ corresponds to a permutohedron move or both $(\lambda_0,\varrho_0)\lessdot (\lambda_1,\varrho_1)$ and $(\lambda_0,\varrho_0)\lessdot(\lambda_2,\varrho_2)$ correspond to associahedron moves. 

\medskip 

\noindent{\bf Case 1.} Assume that both of the cover relations $(\lambda_0,\varrho_0)\lessdot (\lambda_1,\varrho_1)$ and $(\lambda_0,\varrho_0)\lessdot(\lambda_2,\varrho_2)$ correspond to permutohedron moves. Then $\varrho_0=\varrho_1=\varrho_2$, and there are vertices $p_1,q_1,p_2,q_2\in\VV^\times$ such that for each $i\in\{1,2\}$, 
\begin{itemize}
\item $p_i$ and $q_i$ are incomparable in $\TT$; 
\item every number in $\varrho_0(p_i)$ is less than every number in $\varrho_0(q_i)$ in $\mathbb Z$; 
\item $\lambda_0\vert_{\varrho_0(p_i)}\lambda_0\vert_{\varrho_0(q_i)}$ is a consecutive factor of $\lambda_0$; 
\item $\lambda_i$ is obtained from $\lambda_0$ by swapping $\lambda_0\vert_{\varrho(p_i)}$ and $\lambda_0\vert_{\varrho(q_i)}$. 
\end{itemize}
We will show that $(\lambda_1\vee\lambda_2,\varrho_0)\in\Theta(\TT^\times)$, from which it will follow that $(\lambda_1\vee\lambda_2,\varrho_0)$ is the least upper bound of $(\lambda_1,\varrho_1)$ and $(\lambda_2,\varrho_2)$. If $p_1,q_1,p_2,q_2$ are pairwise distinct, then $\lambda_1\vee\lambda_2$ is obtained from $\lambda_0$ by swapping $\lambda_0\vert_{\varrho(p_1)}$ and $\lambda_0\vert_{\varrho(q_1)}$ and also swapping $\lambda_0\vert_{\varrho(p_2)}$ and $\lambda_0\vert_{\varrho(q_2)}$ (these two swaps commute with each other). In this case, it is straightforward to see that $(\lambda_1\vee\lambda_2,\varrho_0)\in\Theta(\TT^\times)$. Now suppose $p_1,q_1,p_2,q_2$ are not pairwise distinct. Then either $p_2=q_1$ or $p_1=q_2$; we may assume without loss of generality that $p_2=q_1$. The linear extension $\lambda_0$ has $\lambda_0\vert_{\varrho_0(p_1)}\lambda_0\vert_{\varrho_0(q_1)}\lambda_0\vert_{\varrho_0(q_2)}$ as a consecutive factor, and $\lambda_1\vee\lambda_2$ is obtained from $\lambda_0$ by replacing this consecutive factor with $\lambda_0\vert_{\varrho_0(q_2)}\lambda_0\vert_{\varrho_0(q_1)}\lambda_0\vert_{\varrho_0(p_1)}$. In this case, we once again see that $(\lambda_1\vee\lambda_2,\varrho_0)\in\Theta(\TT^\times)$. 

\medskip 

\noindent{\bf Case 2.} Assume that both of the cover relations $(\lambda_0,\varrho_0)\lessdot (\lambda_1,\varrho_1)$ and $(\lambda_0,\varrho_0)\lessdot(\lambda_2,\varrho_2)$ correspond to associahedron moves. Then $\lambda_0=\lambda_1=\lambda_2$, and for each $i\in\{1,2\}$, there is a vertex $t_i\in\VV^\times$ such that $\varrho_0(v)=\varrho_i(v)$ for all $v\in\VV^\times\setminus\{t_i\}$ and $\varrho_i(t_i)=\varrho_0(t_i)\cup\varrho_0(t_i^\to)$, where $t_i^\to$ is the vertex in $\VV^\times$ appearing immediately after $\lambda_0\vert_{\varrho_0(t_i)}$ in $\lambda_0$. We will show that $(\lambda_0,\varrho_1\vee\varrho_2)\in\Theta(\TT^\times)$, from which it will follow that $(\lambda_0,\varrho_1\vee\varrho_2)$ is the least upper bound of $(\lambda_1,\varrho_1)$ and $(\lambda_2,\varrho_2)$. 
If $t_1,t_1^\to,t_2,t_2^\to$ are pairwise distinct, then the ornamentation $\varrho_1\vee\varrho_2$ sends $t_1$ to $\varrho_0(t_1)\cup\varrho_0(t_1^\to)$, sends $t_2$ to $\varrho_0(t_2)\cup\varrho_0(t_2^\to)$, and sends each vertex $v\in\VV^\times\setminus\{t_1,t_2\}$ to $\varrho_0(v)$. In this case, it is straightforward to see that $(\lambda_0,\varrho_1\vee\varrho_2)\in\Theta(\TT^\times)$. Now suppose $t_1,t_1^\to,t_2,t_2^\to$ are not pairwise distinct. Then either $t_2=t_1^\to$ or $t_1=t_2^\to$; we may assume without loss of generality that $t_2=t_1^\to$. The ornamentation $\varrho_1\vee\varrho_2$ sends $t_1$ to $\varrho_0(t_1)\cup\varrho_0(t_2)\cup\varrho_0(t_2^\to)$, sends $t_2$ to $\varrho_0(t_2)\cup\varrho_0(t_2^\to)$, and sends each vertex $v\in\VV^\times\setminus\{t_1,t_2\}$ to $\varrho_0(v)$. In this case, we once again see that $(\lambda_0,\varrho_1\vee\varrho_2)\in\Theta(\TT^\times)$. 

\medskip 

\noindent {\bf Case 3.} Assume that the cover relation $(\lambda_0,\varrho_0)\lessdot(\lambda_1,\varrho_1)$ corresponds to a permutohedron move while $(\lambda_0,\varrho_0)\lessdot(\lambda_2,\varrho_2)$ corresponds to an associahedron move. Then $\varrho_0=\varrho_1$, and there are vertices $p_1,q_1\in\VV^\times$ such that 
\begin{itemize}
\item $p_1$ and $q_1$ are incomparable in $\TT$; 
\item every number in $\varrho_0(p_1)$ is less than every number in $\varrho_0(q_1)$ in $\mathbb Z$; 
\item $\lambda_0\vert_{\varrho_0(p_1)}\lambda_0\vert_{\varrho_0(q_1)}$ is a consecutive factor of $\lambda_0$; 
\item $\lambda_1$ is obtained from $\lambda_0$ by swapping $\lambda_0\vert_{\varrho(p_1)}$ and $\lambda_0\vert_{\varrho(q_1)}$. 
\end{itemize}
Furthermore, $\lambda_0=\lambda_2$, and there is a vertex $t_2$ such that $\varrho_0(v)=\varrho_2(v)$ for all $v\in\VV^\times\setminus\{t_2\}$ and $\varrho_2(t_2)=\varrho_0(t_2)\cup\varrho_0(t_2^\to)$, where $t_2^\to$ is the vertex in $\VV^\times$ appearing immediately after $\lambda_0\vert_{\varrho_0(t_2)}$ in $\lambda_0$. If $p_1,q_1,t_2,t_2^\to$ are pairwise distinct, then $(\lambda_1,\varrho_2)$ is the least upper bound of $(\lambda_1,\varrho_1)$ and $(\lambda_2,\varrho_2)$. Now assume $p_1,q_1,t_2,t_2^\to$ are not pairwise distinct. Then either $q_1=t_2$ or $p_1=t_2^\to$. 

Suppose first that $q_1=t_2$. We have $p_1<q_2$, and $p_1$ and $q_2$ are incomparable in $\TT$. It follows that $p_1$ is less than every number in $\varrho_2(t_2)$ in $\mathbb Z$. Note that $\lambda_0\vert_{\varrho_0(p_1)}\lambda_0\vert_{\varrho_2(t_2)}$ is a consecutive factor of $\lambda_0$. Let $\lambda'$ be the linear extension of $\TT^\times$ obtained from $\lambda_0$ by swapping $\lambda_0\vert_{\varrho_0(p_1)}$ and $\lambda_0\vert_{\varrho_2(t_2)}$. Noting that $\lambda_1\leq\lambda'$, we find that $(\lambda',\varrho_2)$ is in $\Theta(\TT^\times)$ and is an upper bound of $(\lambda_1,\varrho_1)$ and $(\lambda_2,\varrho_2)$. We claim that $(\lambda',\varrho_2)$ is actually the least upper bound of $(\lambda_1,\varrho_1)$ and $(\lambda_2,\varrho_2)$. To see this, suppose $(\widehat\lambda,\widehat\varrho)\in\Theta(\TT^\times)$ is an upper bound of $(\lambda_1,\varrho_1)$ and $(\lambda_2,\varrho_2)$; we will show that $\lambda'\leq\widehat\lambda$ so that $(\lambda',\varrho_2)\leq(\widehat\lambda,\widehat\varrho)$. Because $\lambda_1\leq \widehat\lambda$, we must have $\varrho_0(t_2)\prec_{\widehat\lambda}\varrho_0(p_1)$. Because $p_1$ and $t_2$ are incomparable in $\TT$, none of the elements of $\varrho_0(p_1)$ can belong to $\widehat \varrho(t_2)$. On the other hand, since $\varrho_2(t_2)\subseteq\widehat\varrho(t_2)$ and $(\widehat\lambda,\widehat\varrho)\in\Theta(\TT^\times)$, we must have $\varrho_2(t_2)\prec_{\widehat\lambda}\varrho_0(p_1)$. It follows that $\lambda'\leq\widehat\lambda$, as desired. 

Now suppose $p_1=t_2^\to$ and $t_2\not\leq_\TT q_1$. Because $t_2^\to\in\varrho_2(t_2)$, we have $t_2<t_2^\to=p_1<q_1$. Note that $\lambda_0\vert_{\varrho_2(t_2)}\lambda_0\vert_{\varrho_0(q_1)}$ is a consecutive factor of $\lambda_0$. All of the elements of $\varrho_0(t_2)$ are incomparable to all of the elements of $\varrho_0(q_1)$ in $\TT$, and an argument similar to the one used in the preceding paragraph shows that $(\lambda',\varrho_2)$ is the least upper bound of $(\lambda_1,\varrho_1)$ and $(\lambda_2,\varrho_2)$, where $\lambda'$ is the linear extension of $\TT^\times$ obtained from $\lambda_0$ by swapping $\lambda_0\vert_{\varrho_2(t_2)}$ and $\lambda_0\vert_{\varrho_0(q_1)}$.  

Finally, suppose $p_1=t_2^\to$ and $t_2\leq_\TT q_1$. Let $\varrho'$ be the ornamentation of $\TT^\times$ such that $\varrho'(v)=\varrho_0(v)$ for all $v\in\VV^\times\setminus\{t_2\}$ and $\varrho'(t_2)=\varrho_0(t_2)\cup\varrho_0(p_1)\cup\varrho_0(q_1)$. Noting that $\varrho_2\leq\varrho'$, we find that $(\lambda_1,\varrho')$ is in $\Theta(\TT^\times)$ and is an upper bound of $(\lambda_1,\varrho_1)$ and $(\lambda_2,\varrho_2)$. We claim that $(\lambda_1,\varrho')$ is actually the least upper bound of $(\lambda_1,\varrho_1)$ and $(\lambda_2,\varrho_2)$. To see this, suppose $(\widehat\lambda,\widehat\varrho)\in\Theta(\TT^\times)$ is an upper bound of $(\lambda_1,\varrho_1)$ and $(\lambda_2,\varrho_2)$; we will show that $\varrho'\leq\widehat\varrho$ so that $(\lambda_1,\varrho')\leq(\widehat\lambda,\widehat\varrho)$. To prove that $\varrho'\leq\widehat\varrho$, we just need to demonstrate that $\varrho'(t_2)\subseteq\widehat\varrho(t_2)$. The assumption that $\varrho_2\leq\widehat\varrho$ already tells us that $\varrho_0(t_2)\cup\varrho_0(p_1)=\varrho_2(t_2)\subseteq\widehat\varrho(t_2)$. Hence, we just need to show that $\varrho_0(q_1)\subseteq\widehat\varrho(t_2)$. Because $t_2\leq_\TT q_1$, we have $\{t_2\}\prec_{\widehat\lambda}\varrho_0(q_1)$. Since $\lambda_1\leq\widehat\lambda$, we also know that $\varrho_0(q_1)\prec_{\widehat\lambda}\{p_1\}$. The elements of $\widehat\varrho(t_2)$ form a consecutive factor of $\widehat\lambda$, and $t_2,p_1\in\widehat\varrho(t_2)$. This shows that $\varrho_0(q_1)\subseteq\widehat\varrho(t_2)$, as desired. 
\end{proof}

\section{Intervals}\label{sec:intervals}
The purpose of this brief section is to prove \cref{prop:interval}, which states that if $\TT$ contains $\TT'$, then $\MN(\TT')$ is isomorphic to an interval of $\MN(\TT)$. 

\begin{proof}[Proof of \cref{prop:interval}]
It suffices to prove that $\MN(\TT')$ is isomorphic to an interval of $\MN(\TT)$ when $\TT'$ is obtained from $\TT$ by contracting a single edge $e$. Let $u$ and $u'$ be the bottom and top vertices of $e$, respectively. Let $Q$ be the subposet of $\MN(\TT)$ consisting of the maximal nestings of $\TT$ that contain the tube $\tau^*=\{u,u'\}$. The map $\NN\mapsto\NN\cup\{\tau^*\}$ is a poset isomorphism from $\MN(\TT')$ to $Q$, so $Q$ has a unique minimal element $\NN_{\min}$ and a unique maximal element $\NN_{\max}$. Suppose $\NN\in\MN(\TT)$ is such that $\NN_{\min}<\NN<\NN_{\max}$; we will prove that $\NN\in Q$. 

Let $\TT^\times$ be the forest poset obtained by deleting the root from $\TT$, and recall the poset isomorphism $\Psi\colon\MN(\TT)\to\Theta(\TT^\times)$ from \cref{prop:isomorphism}. Note that $u$ appears immediately before $u'$ in both $\lambda_{\NN_{\min}}$ and $\lambda_{\NN_{\max}}$. Note also that $u'$ belongs to both $\varrho_{\NN_{\min}}(u)$ and $\varrho_{\NN_{\max}}(u)$ and that ${\varrho_{\NN_{\min}}(u') = \varrho_{\NN_{\max}}(u') = \{u'\}}$. Because \[\Psi(\NN_{\min})<\Psi(\NN)<\Psi(\NN_{\max}),\] we have \[\lambda_{\NN_{\min}}<\lambda_\NN<\lambda_{\NN_{\max}}\quad\text{and}\quad\varrho_{\NN_{\min}}<\varrho_\NN<\varrho_{\NN_{\max}}.\] This implies that $u$ appears immediately before $u'$ in $\lambda_\NN$, that $u'\in\varrho_\NN(u)$, and that $\varrho_{\NN}(u')=\{u'\}$. It follows from the recursive description of $\Psi^{-1}$ in the proof of \cref{lem:bijection_Psi} that $\tau^*\in\NN$; that is, $\NN\in Q$. 
\end{proof}

\section{Semidistributivity}\label{sec:semidistributivity}

We now proceed to characterize the rooted plane trees $\TT$ such that $\MN(\TT)$ is semidistributive. Our proof will make use of the following result due to Barnard. 

\begin{proposition}[{\cite{Barnard}*{Proposition~22}}]\label{prop:Barnard}\
\begin{itemize}
\item A finite lattice $L$ is meet-semidistributive if and only if for every cover relation $x\lessdot y$ in $L$, the set $\nabla_L(x)\setminus\nabla_L(y)$ has a unique maximal element. 
\item A finite lattice $L$ is join-semidistributive if and only if for every cover relation $x\lessdot y$ in $L$, the set $\Delta_L(y)\setminus\Delta_L(x)$ has a unique minimal element. 
\end{itemize}
\end{proposition}

Throughout this section, we will aim to apply \cref{prop:Barnard} with $L=\Theta(\TT^\times)$, where $\TT\in\PT_n$. We identify the vertex set $\VV$ of $\TT$ with $\{0,1,\ldots,n\}$ so that $0,1,\ldots,n$ is the preorder traversal of $\TT$. Recall that we write $\VV^\times$ for the vertex set of $\TT^\times$ (so $\VV^\times=[n]$). For $(\lambda,\varrho)\in\Theta(\TT^\times)$, we will ease notation by writing $\Delta(\lambda,\varrho)$ and $\nabla(\lambda,\varrho)$ instead of $\Delta_{\Theta(\TT^\times)}(\lambda,\varrho)$ and $\nabla_{\Theta(\TT^\times)}(\lambda,\varrho)$. 
We will separate the main pieces of our argument into four lemmas. The reader may find it helpful to consult \cref{exam:semidistributivity1} while reading the proofs of \cref{lem:piece_1,lem:piece_2} and to consult \cref{exam:semidistributivity2} while reading the proofs of \cref{lem:piece_3,lem:piece_4}.

\begin{lemma}\label{lem:piece_1}
Let $(\lambda_1,\varrho_1)\lessdot(\lambda_2,\varrho_2)$ be a cover relation in $\Theta(\TT^\times)$. If $(\lambda_1,\varrho_1)$ and $(\lambda_2,\varrho_2)$ are related by an associahedron move, then the set $\nabla(\lambda_1,\varrho_1)\setminus\nabla(\lambda_2,\varrho_2)$ has a unique maximal element. 
\end{lemma}

\begin{proof}
Because $(\lambda_1,\varrho_1)$ and $(\lambda_2,\varrho_2)$ are related by an associahedron move, we have $\lambda_1=\lambda_2$, and there is a vertex $t\in \VV^\times$ such that $\varrho_1(v)=\varrho_2(v)$ for all $v\in\VV^\times\setminus\{t\}$ and $\varrho_2(t)=\varrho_1(t)\cup\varrho_1(t^\to)$, where $t^\to$ is the vertex in $\VV^\times$ appearing immediately after the elements of $\varrho_1(t)$ in $\lambda_1$. Let \[A=\{a\in\nabla_{\TT}(t)\setminus\nabla_{\TT}(t^\to):(a,t^\to)\not\in\Inv(\lambda_1)\}.\] Define an ornamentation $\varrho^*\in\OO(\TT^\times)$ by letting $\varrho^*(a)=\nabla_\TT(a)\cap A$ for all $a\in A$ and letting $\varrho^*(v)=\nabla_\TT(v)$ for all $v\in\VV^\times\setminus A$. 

Let $K$ be the set of pairs $(i,j)$ with $1\leq i<j\leq n$ such that either $i\in A$ and $j\in\nabla_\TT(t^\to)$ or $i\leq_\TT j$. There is a unique linear extension $\lambda^*\in\LL(\TT^\times)$ whose inversions are the pairs $(i,j)$ such that $1\leq i<j\leq n$ and $(i,j)\not\in K$. Note that $\lambda^*$ is the unique maximal element (in the weak order) of the set $\{\sigma\in\LL(\TT^\times):(\sigma,\varrho^*)\in\Theta(\TT^\times)\}$. We will prove that $(\lambda^*,\varrho^*)$ is the unique maximal element of $\nabla(\lambda_1,\varrho_1)\setminus\nabla(\lambda_2,\varrho_2)$. Choose an arbitrary element $(\widehat\lambda,\widehat\varrho)\in\nabla(\lambda_1,\varrho_1)\setminus\nabla(\lambda_2,\varrho_2)$; we will prove that $\widehat\lambda\leq\lambda^*$ and $\widehat\varrho\leq\varrho^*$. 

Let us start by proving that $\widehat\varrho\leq\varrho^*$; to do so, it suffices to show that $\widehat\varrho(a)\subseteq A$ for all $a\in A$. Thus, fix $a\in A$. By the definitions of $t^\to$ and $A$, we must have either $a<t^\to$ and $a\in\varrho_1(t)$ or $a>t^\to$ and $t^\to\prec_{\lambda_1} a$. If $a>t^\to$ and $t^\to\prec_{\lambda_1} a$, then we must have $b>t^\to$ and $t^\to\prec_{\lambda_1} b$ for all $b\in\nabla_\TT(a)$, so $\nabla_\TT(a)\subseteq A$. In this case, we certainly have $\widehat\varrho(a)\subseteq\nabla_\TT(a)\subseteq A$, as desired. Thus, we may assume that $a<t^\to$ and $a\in\varrho_1(t)$. Since $\varrho_1\leq\widehat\varrho$, we have $a\in\varrho_1(t)\subseteq\widehat\varrho(t)$, so $\widehat\varrho(a)\subseteq\widehat\varrho(t)$. Thus, in order to show that $\widehat\varrho(a)\subseteq A$, it suffices to prove that $\widehat\varrho(t)\subseteq A$. As $\lambda_2=\lambda_1\leq\widehat\lambda$ and $(\widehat\lambda,\widehat\varrho)\not\in\nabla(\lambda_2,\varrho_2)$, we know that $\varrho_2\not\leq\widehat\varrho$. We also know that $\varrho_2(v)=\varrho_1(v)\subseteq\widehat\varrho(v)$ for all $v\in\VV^\times\setminus\{t\}$, so we must have $\varrho_1(t)\cup\varrho_1(t^\to)=\varrho_2(t)\not\subseteq\widehat\varrho(t)$. Since $\varrho_1(t)\subseteq\widehat\varrho(t)$, we deduce that $\varrho_1(t^\to)\not\subseteq\widehat\varrho(t)$, so $\widehat\varrho(t^\to)\not\subseteq\widehat\varrho(t)$. This implies that $\widehat\varrho(t^\to)$ and $\widehat\varrho(t)$ are disjoint (since they cannot be nested), so $t^\to\not\in\widehat\varrho(t)$. Hence, $\widehat\varrho(t)\subseteq\nabla_\TT(t)\setminus\nabla_\TT(t^\to)$. Consider an arbitrary vertex $x\in\widehat\varrho(t)$; we will show that $x\in A$. We already know that $x\in\nabla_\TT(t)\setminus\nabla_\TT(t^\to)$, so we just need to prove that $(x,t^\to)\not\in\Inv(\lambda_1)$. Because the elements of $\widehat\varrho(t)$ form a consecutive factor of $\widehat\lambda$ (since $(\widehat\lambda,\widehat\varrho)\in\Theta(\TT^\times)$), we cannot have $t\prec_{\widehat\lambda}t^\to\prec_{\widehat\lambda} x$. Therefore, $(x,t^\to)$ cannot be an inversion of $\widehat\lambda$. Since $\lambda_1\leq\widehat\lambda$, this implies that $(x,t^\to)\not\in\Inv(\lambda_1)$, as desired. 

We now show that $\widehat\lambda\leq\lambda^*$. Suppose instead that $(i,j)\in\Inv(\widehat\lambda)\setminus\Inv(\lambda^*)$. Then $(i,j)\in K$, and we cannot have $i\leq_\TT j$ (since $\widehat\lambda\in\LL(\TT^\times)$), so we must have $i\in A$ and $j\in\nabla_\TT(t^\to)$. This implies that $i<t^\to$ and $(i,t^\to)\not\in\Inv(\lambda_1)$. Since $t<_\TT i$, we deduce that $t\prec_{\lambda_1}i\prec_{\lambda_1}t^\to$. Because $t^\to$ is the vertex appearing immediately after the elements of $\varrho_1(t)$ in $\lambda_1$, this implies that $i\in\varrho_1(t)\subseteq\widehat\varrho(t)$. Because $(i,j)\in\Inv(\widehat\lambda)$ and $t<_\TT j$, we must have $t\prec_{\widehat\lambda}j\prec_{\widehat\lambda}i$. Because $(\widehat\lambda,\widehat\varrho)\in\Theta(\TT^\times)$, the elements of $\widehat\varrho(t)$ form a consecutive factor of $\widehat\lambda$; this consecutive factor includes $t$ and $i$, so $j\in\widehat\varrho(t)$. We saw in the preceding paragraph that $\widehat\varrho(t)\subseteq A$, so $j\in A$. This is a contradiction because $A\subseteq\nabla_\TT(t)\setminus\nabla_\TT(t^\to)$ and $j\in\nabla_\TT(t^\to)$. 
\end{proof}

\begin{lemma}\label{lem:piece_2}
Let $(\lambda_1,\varrho_1)\lessdot(\lambda_2,\varrho_2)$ be a cover relation in $\Theta(\TT^\times)$. If $(\lambda_1,\varrho_1)$ and $(\lambda_2,\varrho_2)$ are related by an associahedron move, then the set $\Delta(\lambda_2,\varrho_2)\setminus\Delta(\lambda_1,\varrho_1)$ has a unique minimal element. 
\end{lemma}

\begin{proof}
Because $(\lambda_1,\varrho_1)$ and $(\lambda_2,\varrho_2)$ are related by an associahedron move, we have $\lambda_1=\lambda_2$, and there is a vertex $t\in \VV^\times$ such that $\varrho_1(v)=\varrho_2(v)$ for all $v\in\VV^\times\setminus\{t\}$ and $\varrho_2(t)=\varrho_1(t)\cup\varrho_1(t^\to)$, where $t^\to$ is the vertex in $\VV^\times$ appearing immediately after the elements of $\varrho_1(t)$ in $\lambda_1$. Define an ornamentation $\varrho_*$ of $\TT^\times$ by letting $\varrho_*(t)=\{t^\to\}\cup\{x\in\varrho_1(t):x<t^\to\}$ and letting $\varrho_*(v)=\{v\}$ for all $v\in\VV^\times\setminus\{t\}$. 

Let $\lambda_*$ be the linear extension $12\cdots (t-1)\mu\nu$ of $\TT^\times$, where $\mu$ is the word obtained by reading the elements of $\varrho_*(t)$ in increasing order and $\nu$ is the word obtained by reading the elements of $\{t,t+1,\ldots,n\}\setminus\varrho_*(t)$ in increasing order. It is straightforward to see that $\lambda_*$ is the unique minimal element (in the weak order) of the set $\{\sigma\in\LL(\TT^\times):(\sigma,\varrho^*)\in\Theta(\TT^\times)\}$. We will prove that $(\lambda_*,\varrho_*)$ is the unique minimal element of $\Delta(\lambda_2,\varrho_2)\setminus\Delta(\lambda_1,\varrho_1)$. Choose an arbitrary element $(\widehat\lambda,\widehat\varrho)\in\Delta(\lambda_2,\varrho_2)\setminus\Delta(\lambda_1,\varrho_1)$; we will prove that $\lambda_*\leq\widehat\lambda$ and $\varrho_*\leq\widehat\varrho$. 

Let us start by proving that $\varrho_*\leq\widehat\varrho$; to do so, it suffices to show that $\varrho_*(t)\subseteq\widehat\varrho(t)$. Thus, let $a\in\varrho_*(t)$.
Because $\widehat\lambda\leq\lambda_2=\lambda_1$, we must have $\widehat\varrho\not\leq\varrho_1$. We have $\widehat\varrho(v)\subseteq\varrho_2(v)=\varrho_1(v)$ for every $v\in\VV^\times\setminus\{t\}$, so $\widehat\varrho(t)\not\subseteq\varrho_1(t)$. But $\widehat\varrho(t)\subseteq\varrho_2(t)=\varrho_1(t)\cup\varrho_1(t^\to)$, so there exists an element $b\in\widehat\varrho(t)\cap\varrho_1(t^\to)$. Since $t<_\TT t^\to\leq_\TT b$ and $\widehat\varrho(t)$ induces a connected subgraph of $\TT^\times$, we must have $t^\to\in\widehat\varrho(t)$. If $a=t^\to$, then we have shown that $a\in\widehat\varrho(t)$, as desired. Now suppose $a\neq t^\to$. By the definition of $\varrho_*(t)$, we know that $a<t^\to$ and that $a\in\varrho_1(t)$. It follows that $(a,t^\to)$ is not an inversion of the linear extension $\lambda_1=\lambda_2$, so (since $\widehat\lambda\leq\lambda_2$) it is also not an inversion of $\widehat\lambda$. This means that $t\preceq_{\widehat\lambda}a\prec_{\widehat\lambda} t^\to$. We know that $t,t^\to\in\widehat\varrho(t)$ and that the elements of $\widehat\varrho(t)$ form a consecutive factor of $\widehat\lambda$, so $a$ must also be in $\widehat\varrho(t)$. 

Let us now prove that $\lambda_*\leq\widehat\lambda$. Suppose $(i,j)$ is an inversion of $\lambda_*$. Then $j$ must be in the word $\mu$, while $i$ must be in the word $\nu$. In other words, we have $j\in\varrho_*(t)$ and $i\in\{t,t+1,\ldots,n\}\setminus\varrho_*(t)$. We saw in the preceding paragraph that $\varrho_*(t)\subseteq\widehat\varrho(t)$, so $j\in\widehat\varrho(t)$. We have $i<j\leq t^\to$, so it follows from the fact that $i\not\in\varrho_*(t)$ that $i\not\in\varrho_1(t)$. Moreover, the fact that $i<t^\to$ implies that $i\not\in\varrho_1(t^\to)$. Therefore, $i\not\in\varrho_2(t)$. As $\widehat\varrho(t)\subseteq\varrho_2(t)$, we find that $i\not\in\widehat\varrho(t)$. Thus, we have shown that $\widehat\varrho(t)$ contains $t$ and $j$ but not $i$. We also know that $t<i<j$ and that $t<_\TT j$, so it follows from the definition of the preorder traversal that $t<_\TT i$. This implies that $t\prec_{\widehat\lambda}i$. However, $\widehat\varrho(t)$ is a consecutive factor of $\widehat\lambda$ that contains $t$ and $j$ but not $i$, so $j\prec_{\widehat\lambda}i$. This shows that $(i,j)$ is an inversion of $\widehat\lambda$. As $(i,j)$ was an arbitrary inversion of $\lambda_*$, we deduce that $\lambda_*\leq\widehat\lambda$. 

The preceding paragraphs show that every element of ${\Delta(\lambda_2,\varrho_2)\setminus\Delta(\lambda_1,\varrho_1)}$ is greater than or equal to $(\lambda_*,\varrho_*)$. In particular, $(\lambda_*,\varrho_*)\in\Delta(\lambda_2,\varrho_2)$. Because $t^\to\in\varrho_*(t)\setminus\varrho_1(t)$, we have $\varrho_*\not\leq\varrho_1$. It follows that $(\lambda_*,\varrho_*)\not\in\Delta(\lambda_1,\varrho_1)$. This shows that $(\lambda_*,\varrho_*)$ is actually in the set $\Delta(\lambda_2,\varrho_2)\setminus\Delta(\lambda_1,\varrho_1)$, so it must be the unique minimal element of this set. 
\end{proof}

\begin{example}\label{exam:semidistributivity1}
\cref{fig:exam:semidistributivity1} portrays a cover relation $(\lambda_1,\varrho_1)\lessdot(\lambda_2,\varrho_2)$ in $\Theta(\TT^\times)$, where $\TT^\times$ is as depicted. This cover relation corresponds to an associahedron move. In the notation of the proofs of \cref{lem:piece_1,lem:piece_2}, we have $t=4$ and $t^\to=7$. The pairs $(\lambda^*,\varrho^*)$ and $(\lambda_*,\varrho_*)$ constructed in the proofs of \cref{lem:piece_1,lem:piece_2} are shown on the left and right, respectively, in \cref{fig:exam:semidistributivity1_2}. 
\end{example}

\begin{figure}[ht]
\begin{center}
\includegraphics[height=15.157cm]{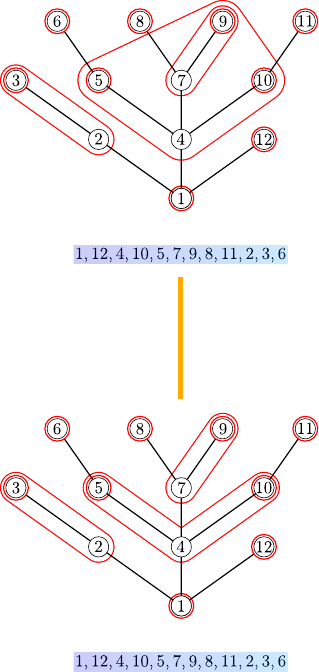}
\end{center}
  \caption{A cover relation $(\lambda_1,\varrho_1)\lessdot(\lambda_2,\varrho_2)$ corresponding to an associahedron move.}\label{fig:exam:semidistributivity1}
\end{figure}

\begin{figure}[ht]
\begin{center}
\includegraphics[height=5.958cm]{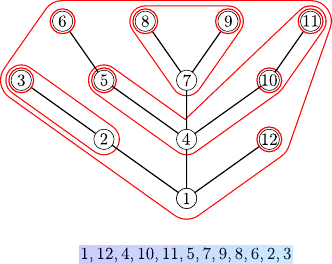}\qquad \includegraphics[height=5.771cm]{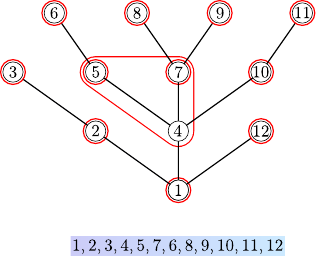}
\end{center}
  \caption{The pairs $(\lambda^*,\varrho^*)$ (left) and $(\lambda_*,\varrho_*)$ (right) constructed in the proofs of \cref{lem:piece_1,lem:piece_2}, where $(\lambda_1,\varrho_1)\lessdot(\lambda_2,\varrho_2)$ is the cover relation shown in \cref{fig:exam:semidistributivity1}. }\label{fig:exam:semidistributivity1_2}
\end{figure}

\begin{lemma}\label{lem:piece_3}
Suppose that every vertex of $\TT$ that is not in the rightmost branch of $\TT$ is covered by at most $1$ element of $\TT$. Let $(\lambda_1,\varrho_1)\lessdot(\lambda_2,\varrho_2)$ be a cover relation in $\Theta(\TT^\times)$. If $(\lambda_1,\varrho_1)$ and $(\lambda_2,\varrho_2)$ are related by a permutohedron move, then the set $\nabla(\lambda_1,\varrho_1)\setminus\nabla(\lambda_2,\varrho_2)$ has a unique maximal element. 
\end{lemma}

\begin{proof}
Let $\mathsf{B}$ be the rightmost branch of $\TT$. Because $(\lambda_1,\varrho_1)$ and $(\lambda_2,\varrho_2)$ are related by a permutohedron move, we have $\varrho_1=\varrho_2$, and there exist vertices $p,q\in\VV^\times$ such that 
\begin{itemize}
\item $p$ and $q$ are incomparable in $\TT$; 
\item every number in $\varrho_1(p)$ is less than every number in $\varrho_1(q)$ in $\mathbb Z$; 
\item $\lambda_1\vert_{\varrho_1(p)}\lambda_1\vert_{\varrho_1(q)}$ is a consecutive factor of $\lambda_1$; 
\item $\lambda_2$ is obtained from $\lambda_1$ by swapping $\lambda_1\vert_{\varrho_1(p)}$ and $\lambda_1\vert_{\varrho_1(q)}$. 
\end{itemize}
The vertex $p$ cannot be in $\mathsf{B}$, so the sets $\varrho_1(p)$ and $C=\{v\in\VV^\times\setminus\mathsf{B}:v\leq_\TT p\}$ are both chains in $\TT$. 
Let $u$ be the maximum element of $\varrho_1(p)$, and let $z$ be the minimum element of $C$. Let \[X=\{x\in\VV^\times:z\leq x<q\text{ and }(x,q)\not\in\Inv(\lambda_1)\}\] and \[Y=\{y\in\VV^\times:z\leq y<q\text{ and }(y,q)\in\Inv(\lambda_1)\}.\] Because the elements of $\varrho_1(p)$ occur immediately before the elements of $\varrho_1(q)$ in $\lambda_1$, we have  \begin{equation}\label{eq:zY}
\nabla_\TT(u)\setminus\{u\}\subseteq Y.
\end{equation} Let $\varrho^*\in\OO(\TT^\times)$ be the ornamentation such that $\varrho^*(x)=\nabla_\TT(x)\setminus(\nabla_\TT(q)\cup Y)$ for all $x\in X$ and $\varrho^*(w)=\nabla_\TT(w)$ for all $w\in \VV^\times\setminus X$. Let \[\Xi=\{\sigma\in\LL(\TT^\times):X\prec_{\sigma}\varrho_1(q)\prec_{\sigma}Y\}.\] In the weak order, the set $\Xi$ has a unique maximal element $\lambda^*$. Moreover, $(\lambda^*,\varrho^*)\in\Theta(\TT^\times)$. We will prove that $(\lambda^*,\varrho^*)$ is the unique maximal element of $\nabla(\lambda_1,\varrho_1)\setminus\nabla(\lambda_2,\varrho_2)$. Choose an arbitrary element $(\widehat\lambda,\widehat\varrho)\in\nabla(\lambda_1,\varrho_1)\setminus\nabla(\lambda_2,\varrho_2)$; we will prove that $\widehat\lambda\leq\lambda^*$ and $\widehat\varrho\leq\varrho^*$. 

Let us start by proving that $\widehat\lambda\leq\lambda^*$; to do so, it suffices to show that $\widehat\lambda\in\Xi$. Because $\varrho_1=\varrho_2$, we have $\lambda_1\leq\widehat\lambda$ and $\lambda_2\not\leq\widehat\lambda$, so there must be an element of $\varrho_1(p)$ that precedes an element of $\varrho_1(q)$ in $\widehat\lambda$. We have $\varrho_1(p)\subseteq\widehat\varrho(p)$ and $\varrho_1(q)\subseteq\widehat\varrho(q)$, and the sets $\widehat\varrho(p)$ and $\widehat\varrho(q)$ are disjoint because $p$ and $q$ are incomparable in $\TT$. The fact that $(\widehat\lambda,\widehat\varrho)\in\Theta(\TT^\times)$ implies that $\widehat\varrho(p)\prec_{\widehat\lambda}\widehat\varrho(q)$, so $\varrho_1(p)\prec_{\widehat\lambda}\varrho_1(q)$. Consequently, $\{u\}\prec_{\widehat\lambda}\varrho_1(q)$. Let $x\in X$. If $z\leq_\TT x$, then it follows from \eqref{eq:zY} that $x\leq_\TT u$, so $\{x\}\prec_{\widehat\lambda}\varrho_1(q)$. If $z\not\leq_\TT x$, then $(u,x)$ is an inversion of $\lambda_1$, so it is also an inversion of $\widehat\lambda$, and again $\{x\}\prec_{\widehat\lambda}\varrho_1(q)$. This shows that $X\prec_{\widehat\lambda}\varrho_1(q)$.  Now let $y\in Y$. Since $(y,q)$ is an inversion of $\lambda_1$, it is also an inversion of $\widehat\lambda$. Because the elements of $\widehat\varrho(q)$ form a consecutive factor of $\widehat\lambda$, we find that $\widehat\varrho(q)\prec_{\widehat\lambda}\{y\}$. As $\varrho_1(q)\subseteq\widehat\varrho(q)$ and $y$ was an arbitrary element of $Y$, this proves that $\varrho_1(q)\prec_{\widehat\lambda}Y$. Hence, $\widehat\lambda\in\Xi$. 

Let us now prove that $\widehat\varrho\leq\varrho^*$. We have $\widehat\varrho(w)\subseteq\nabla_\TT(w)=\varrho^*(w)$ for all $w\in\VV^\times\setminus X$. Now consider $x\in X$; we will prove that $\widehat\varrho(x)\subseteq\nabla_\TT(x)\setminus(\nabla_\TT(q)\cup Y)=\varrho^*(x)$. The elements of $\widehat\varrho(x)$ form a consecutive factor of $\widehat\lambda$, and we know from the preceding paragraph that $\{x\}\prec_{\widehat\lambda}\varrho_1(q)\prec_{\widehat\lambda} Y$, so it suffices to show that $q\not\in\widehat\varrho(x)$. This is immediate if $x\not\leq_\TT q$, so assume $x\leq_\TT q$. Then $x\prec_{\lambda_1}q$. By the definition of $X$, we must have $x\not\leq_\TT p$, so $p\not\in\widehat\varrho(x)$. Since $\lambda_{1}\vert_{\varrho_1(p)}\lambda_1\vert_{\varrho_1(q)}$ is a consecutive factor of $\lambda_1$ and $x\not\in\varrho_1(p)$, we must have $(p,x)\in\Inv(\lambda_1)\subseteq\Inv(\widehat\lambda)$. This implies that $\widehat\varrho(x)\prec_{\widehat\lambda} \{p\}\prec_{\widehat\lambda} \{q\}$, so $q\not\in\widehat\varrho(x)$. 
\end{proof}

\begin{lemma}\label{lem:piece_4}
Suppose that every vertex of $\TT$ that is not in the rightmost branch of $\TT$ is covered by at most $1$ element of $\TT$. Let $(\lambda_1,\varrho_1)\lessdot(\lambda_2,\varrho_2)$ be a cover relation in $\Theta(\TT^\times)$. If $(\lambda_1,\varrho_1)$ and $(\lambda_2,\varrho_2)$ are related by a permutohedron move, then the set $\Delta(\lambda_2,\varrho_2)\setminus\Delta(\lambda_1,\varrho_1)$ has a unique minimal element. 
\end{lemma}

\begin{proof}
Because $(\lambda_1,\varrho_1)$ and $(\lambda_2,\varrho_2)$ are related by a permutohedron move, we have $\varrho_1=\varrho_2$, and there exist vertices $p,q\in\VV^\times$ such that 
\begin{itemize}
\item $p$ and $q$ are incomparable in $\TT$; 
\item every number in $\varrho_1(p)$ is less than every number in $\varrho_1(q)$ in $\mathbb Z$; 
\item $\lambda_1\vert_{\varrho_1(p)}\lambda_1\vert_{\varrho_1(q)}$ is a consecutive factor of $\lambda_1$; 
\item $\lambda_2$ is obtained from $\lambda_1$ by swapping $\lambda_1\vert_{\varrho_1(p)}$ and $\lambda_1\vert_{\varrho_1(q)}$. 
\end{itemize}
The vertex $p$ cannot belong to the rightmost branch of $\TT$, so $\varrho_1(p)$ (viewed as a subposet of $\TT$) is a chain. Let $u$ be the maximum element of $\varrho_1(p)$. Let \[Z_{\mathsf{L}}=\{i\in\VV^\times: u\leq i\leq q \text{ and }i\preceq_{\lambda_2}q\}\quad\text{and}\quad Z_{\mathsf{R}}=\{i\in\VV^\times: u\leq i\leq q \text{ and }u\preceq_{\lambda_2}i\}.\]
Let $\zeta_{\mathsf{L}}$ (respectively, $\zeta_{\mathsf{R}}$) be the word obtained by writing the elements of $Z_{\mathsf{L}}$ (respectively, $Z_{\mathsf{R}}$) in increasing order. Let $\lambda_*$ be the permutation \[12\cdots (u-1)\,\zeta_{\mathsf{L}}\,\zeta_{\mathsf{R}}\,(q+1)\cdots (n-1)n.\] It is straightforward to check that $\lambda_*\leq\lambda_2$. Because $\LL(\TT^\times)$ is an interval in the weak order whose minimum element is the identity permutation, it must contain $\lambda_*$. Recall that the unique minimal element $\varrho_{\min}$ of $\OO(\TT^\times)$ is defined so that $\varrho_{\min}(v)=\{v\}$ for all $v\in\VV^\times$. Note that $(\lambda_*,\varrho_{\min})\in\Theta(\TT^\times)$ and that $(\lambda_*,\varrho_{\min})\in\Delta(\lambda_2,\varrho_2)$. Because $(u,q)\in\Inv(\lambda_*)\setminus\Inv(\lambda_1)$, we have $(\lambda_*,\varrho_{\min})\in\Delta(\lambda_2,\varrho_2)\setminus\Delta(\lambda_1,\varrho_1)$. We will prove that $(\lambda_*,\varrho_{\min})$ is the unique minimal element of $\Delta(\lambda_2,\varrho_2)\setminus\Delta(\lambda_1,\varrho_1)$. Choose an arbitrary element $(\widehat\lambda,\widehat\varrho)\in\Delta(\lambda_2,\varrho_2)\setminus\Delta(\lambda_1,\varrho_1)$; we know already that $\varrho_{\min}\leq\widehat\varrho$, so we just need to prove that $\lambda_*\leq\widehat\lambda$. 

Let $(i,j)$ be an inversion of $\lambda_*$; our goal is to show that $(i,j)$ is also an inversion of $\widehat\lambda$. Because $\varrho_1=\varrho_2$, we must have $\widehat\lambda\leq\lambda_2$ and $\widehat\lambda\not\leq\lambda_1$. This means that there is an inversion $(a,b)$ of $\widehat\lambda$ that is also an inversion of $\lambda_2$ but not of $\lambda_1$. We must have $a\in\varrho_1(p)$ and $b\in\varrho_1(q)$. Then $a\leq_\TT u$ and $q\leq_\TT b$, so $a\preceq_{\widehat\lambda}u$ and $q\preceq_{\widehat\lambda}b$. Because $(i,j)\in\Inv(\lambda_*)$, we must have $i\in Z_{\mathsf{R}}$ and $j\in Z_{\mathsf{L}}$. This implies that $(u,i)$ and $(j,q)$ are not inversions of $\lambda_2$, so they are also not inversions of $\widehat\lambda$. Hence, $j\preceq_{\widehat\lambda}q\preceq_{\widehat\lambda}b\prec_{\widehat\lambda}a\preceq_{\widehat\lambda}u\preceq_{\widehat\lambda}i$. This shows that $(i,j)$ is an inversion of $\widehat\lambda$, as desired.  
\end{proof}

\begin{example}\label{exam:semidistributivity2}
Let $\varrho_1=\varrho_2\in\OO(\TT^\times)$ be the ornamentation shown in \cref{fig:exam:semidistributivity2}, where $\TT^\times$ is as depicted. In the tree $\TT$ (which is obtained by adding a new root vertex to $\TT^\times$), every vertex that is not in the rightmost branch is covered by at most $1$ element. Let \[\lambda_1=1,2,4,8,9,12,13,18,22,23,25,24,16,14,5,6,19,20,17,21,3,7,15,10,11\] and 
\[\lambda_2=1,2,4,8,9,12,13,18,22,23,25,24,16,14,19,20,5,6,17,21,3,7,15,10,11.\] Then $(\lambda_1,\varrho_1)\lessdot(\lambda_2,\varrho_2)$ is a cover relation in $\Theta(\TT^\times)$ that corresponds to a permutohedron move. In the notation of the proofs of \cref{lem:piece_3,lem:piece_4}, we have $p=5$, $q=19$, $u=6$, and $z=4$. The sets $X$ and $Y$ defined in the proof of \cref{lem:piece_3} are 
\[X=\{4,5,6,8,9,12,13,14,16,18\}\quad\text{and}\quad Y=\{7,10,11,15,17\}.\] The sets $Z_{\mathsf{L}}$ and $Z_{\mathsf{R}}$ defined in the proof of \cref{lem:piece_4} are 
\[Z_{\mathsf{L}}=\{8,9,12,13,14,16,18,19\}\quad\text{and}\quad Z_{\mathsf{R}}=\{6,7,10,11,15,17\}.\] The pairs $(\lambda^*,\varrho^*)$ and $(\lambda_*,\varrho_{\min})$ constructed in the proofs of \cref{lem:piece_3,lem:piece_4} are shown on the top and bottom, respectively, in \cref{fig:exam:semidistributivity2_2}. In the top image in \cref{fig:exam:semidistributivity2_2}, the sets $X$ and $Y$ are represented in green and purple, respectively. 
\end{example}

\begin{figure}[ht]
\begin{center}
\includegraphics[height=9.070cm]{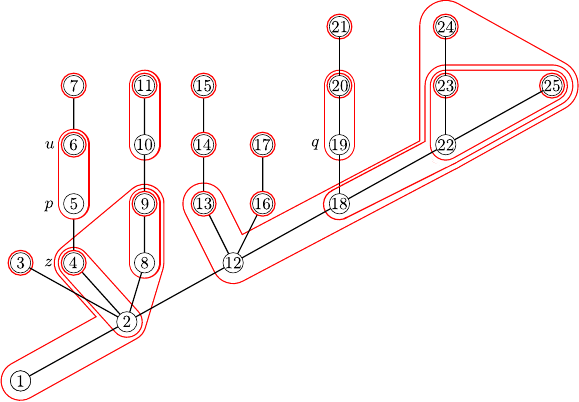}
\end{center}
  \caption{An ornamentation of a forest poset $\TT^\times$. Every vertex of $\TT$ that is not in the rightmost branch of $\TT$ is covered by at most $1$ element.  }\label{fig:exam:semidistributivity2}
\end{figure}

\begin{figure}[]
\begin{center}
\includegraphics[height=10.168cm]{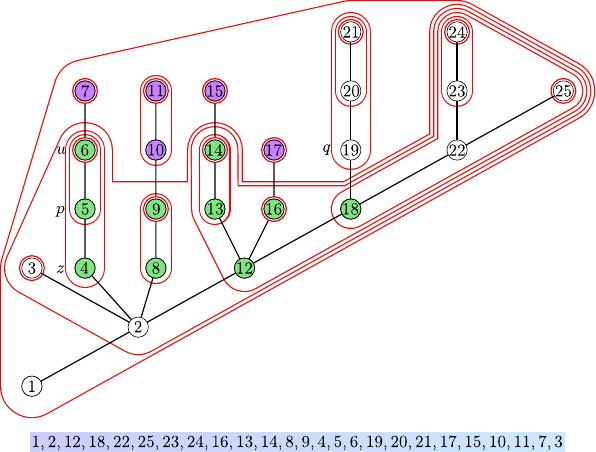}\end{center}
\vspace{0.4cm}
\begin{center}
\includegraphics[height=9.771cm]{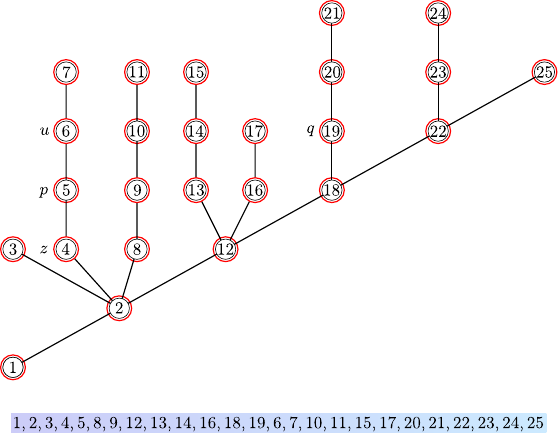}
\end{center}
  \caption{The pairs $(\lambda^*,\varrho^*)$ (top) and $(\lambda_*,\varrho_{\min})$ (bottom) constructed in the proofs of \cref{lem:piece_3,lem:piece_4}, where $(\lambda_1,\varrho_1)\lessdot(\lambda_2,\varrho_2)$ is the cover relation defined in \cref{exam:semidistributivity2} (with $\varrho_1=\varrho_2$ appearing in \cref{fig:exam:semidistributivity2}). In the top image, the elements of $X$ are colored green, while the elements of $Y$ are colored purple.}\label{fig:exam:semidistributivity2_2}
\end{figure}

We can now tie together all of the pieces established so far in this section to prove \cref{thm:semidistributive}. 

\begin{proof}[Proof of \cref{thm:semidistributive}]
Recall that our goal is to prove the equivalence of the five items \ref{S1}, \ref{S2}, \ref{S3}, \ref{S4}, \ref{S5} listed in the statement of the theorem. By the definition of semidistributivity, item \ref{S1} holds if and only if items \ref{S2} and \ref{S3} both hold. It is also straightforward to see that items \ref{S4} and \ref{S5} are equivalent. 

The Hasse diagram of the operahedron lattice of $\begin{array}{l}\includegraphics[height=0.5cm]{OperahedronPIC3}\end{array}$ is   
shown in \cref{fig:operahedron1}. Upon inspecting this figure, we find that the set of maximal nestings $\NN$ of $\begin{array}{l}\includegraphics[height=0.5cm]{OperahedronPIC3}\end{array}$ satisfying 
\[\begin{array}{l}\includegraphics[height=1.5cm]{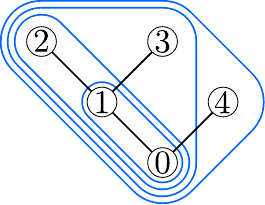}\end{array}\leq\NN\quad\text{and}\quad\begin{array}{l}\includegraphics[height=1.5cm]{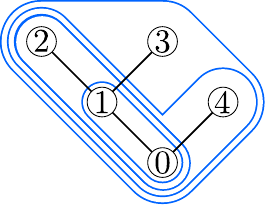}\end{array}\not\leq\NN\]
does not have a unique maximal element. Therefore, it follows from \cref{prop:Barnard} that the operahedron lattice of $\begin{array}{l}\includegraphics[height=0.5cm]{OperahedronPIC3}\end{array}$ is not meet-semidistributive. Similarly, the set of maximal nestings $\NN$ of $\begin{array}{l}\includegraphics[height=0.5cm]{OperahedronPIC3}\end{array}$ satisfying 
\[\NN\leq\begin{array}{l}\includegraphics[height=1.5cm]{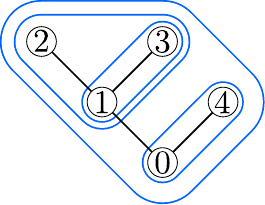}\end{array}\quad\text{and}\quad\NN\not\leq\begin{array}{l}\includegraphics[height=1.5cm]{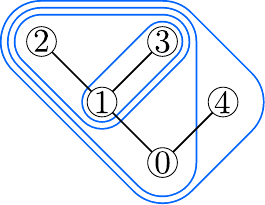}\end{array}\]
does not have a unique minimal element, so it follows from \cref{prop:Barnard} that the operahedron lattice of $\begin{array}{l}\includegraphics[height=0.5cm]{OperahedronPIC3}\end{array}$ is not join-semidistributive. Intervals of meet-semidistributive lattices are meet-semidistributive, so we can appeal to \cref{prop:interval} to see that item \ref{S2} implies item \ref{S4}. Likewise, intervals of join-semidistributive lattices are join-semidistributive, so item \ref{S3} implies item \ref{S4}. 

Finally, it follows from \cref{prop:isomorphism,prop:Barnard,lem:piece_1,lem:piece_2,lem:piece_3,lem:piece_4} that item \ref{S5} implies both items \ref{S2} and \ref{S3}. 
\end{proof} 

\begin{remark}\label{rem:explicit_join-irr}
Suppose $L$ is a semidistributive lattice. One can show that an element $j\in L$ is join-irreducible if and only if there exists a cover relation $x\lessdot y$ such that $j$ is the unique minimal element of $\Delta_L(y)\setminus\Delta_L(x)$. Likewise, an element $m\in L$ is meet-irreducible if and only if there exists a cover relation $x\lessdot y$ such that $m$ is the unique maximal element of $\nabla_L(x)\setminus\nabla_L(y)$.
Therefore, if $\TT$ is a rooted plane tree such that $\MN(\TT)$ (equivalently, $\Theta(\TT^\times)$) is semidistributive, then our proofs of \cref{lem:piece_1,lem:piece_2,lem:piece_3,lem:piece_4} provide explicit descriptions of the join-irreducble elements and the meet-irreducible elements of $\Theta(\TT^\times)$ (and, hence, also of $\MN(\TT)$). 
\end{remark}

\section{Trimness}\label{sec:trim}

We now prove \cref{thm:trim}, which characterizes the rooted plane trees whose operahedron lattices are trim. 

\begin{proof}[Proof of \cref{thm:trim}]
The operahedron lattice of the claw $\begin{array}{l}\includegraphics[height=0.3cm]{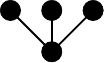}\end{array}$ (see the right side of \cref{fig:Tamari_weak}) is not trim because its height is $3$ and it has $4$ join-irreducible elements. One can check by hand that the operahedron lattices of the trees $\begin{array}{l}\includegraphics[height=0.3cm]{OperahedronPIC20}\end{array}$ and $\begin{array}{l}\includegraphics[height=0.5cm]{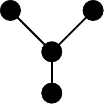}\end{array}$ are isomorphic to each other. It is known \cite{Thomas}*{Theorem~1} that intervals of trim lattices are trim. Therefore, it follows from \cref{prop:interval} that if $\MN(\TT)$ is trim, then $\TT$ does not contain $\begin{array}{l}\includegraphics[height=0.3cm]{OperahedronPIC20}\end{array}$ or $\begin{array}{l}\includegraphics[height=0.5cm]{OperahedronPIC21}\end{array}$. Consequently, if $\MN(\TT)$ is trim, then the root of $\TT$ is covered by at most $2$ elements of $\TT$ and every non-root vertex in $\TT$ is covered by at most $1$ element of $\TT$.  

Now assume that the root of $\TT$ is covered by at most $2$ elements of $\TT$ and that every non-root vertex in $\TT$ is covered by at most $1$ element of $\TT$; we will prove that $\MN(\TT)$ is trim. Let us identify the vertex set of $\TT$ with $\{0,1,\ldots,n\}$ so that $0,1,\ldots,n$ is the preorder traversal of $\TT$. Let $d$ be the largest element of $[n]$ such that $1\leq_\TT d$. If $d=n$, then $\TT$ is a chain, so $\MN(\TT)$ is a Tamari lattice, which is known to be trim. Therefore, we may assume $d\leq n-1$. Then the root of $\TT$ is covered by $1$ and $d+1$. Let $\TT_1$ and $\TT_{d+1}$ be the subtrees of $\TT$ with roots $1$ and $d+1$, respectively. Then $\TT_1$ is a chain with $d$ vertices, while $\TT_{d+1}$ is a chain with $n-d$ vertices. It is known \cite{RiSM}*{Theorem~1.4} that every semidistributive extremal lattice is trim. Since we already know by \cref{thm:semidistributive} that $\MN(\TT)$ is semidistributive, we just need to prove that $\MN(\TT)$ is extremal. It is also known \cite{Free}*{Corollary~2.55} that a semidistributive lattice has the same number of join-irreducible elements as meet-irreducible elements. Therefore, appealing to \cref{prop:isomorphism}, we see that it suffices to prove that $\mathrm{height}(\Theta(\TT^\times))=|\mathcal J_{\Theta(\TT^\times)}|$, where $\mathcal J_{\Theta(\TT^\times)}$ is the set of join-irreducible elements of $\Theta(\TT^\times)$. 

As mentioned in \cref{rem:explicit_join-irr}, an element $(\lambda_*,\varrho_*)\in\Theta(\TT^\times)$ is join-irreducible if and only if it is one of the elements constructed in the proof of \cref{lem:piece_2} or the proof of \ref{lem:piece_4}. Upon inspecting those proofs, we find that $(\lambda_*,\varrho_*)$ is join-irreducible if and only if one of the following (mutually exclusive) conditions holds: 
\begin{enumerate}[(i)]
\item\label{item:trim_proof1} $\varrho_*$ is a join-irreducible element of $\OO(\TT^\times)$, and $\lambda_*$ is the unique minimal element (in the weak order) of the set $\{\sigma\in\LL(\TT^\times):(\sigma,\varrho_*)\in\Theta(\TT^\times)\}$;
\item\label{item:trim_proof2} $\varrho_*=\varrho_{\min}$, and $\lambda_*$ is a join-irreducible element of $\LL(\TT^\times)$.
\end{enumerate}
Note that $\OO(\TT^\times)$ is isomorphic to $\OO(\TT_1)\times\OO(\TT_{d+1})$. Moreover, $\OO(\TT_1)$ (respectively, $\OO(\TT_{d+1})$) is isomorphic to the $d$-th (respectively, $(n-d)$-th) Tamari lattice. It is well known \cite{Geyer}*{Proposition~2.3} that the $m$-th Tamari lattice has $\binom{m}{2}$ join-irreducible elements. Hence, it follows from \eqref{eq:product_join-irr} that the number of pairs $(\lambda_*,\varrho_*)$ satisfying the condition \ref{item:trim_proof1} is $\binom{d}{2}+\binom{n-d}{2}$. The join-irreducible elements of $\LL(\TT^\times)$ are the permutations of the form \[12\cdots a\,(d+1)(d+2)\cdots b\,(a+1)(a+2)\cdots d\,(b+1)(b+2)\cdots n,\] where $0\leq a\leq d-1$ and $d+1\leq b\leq n$. The number of such permutations, which is also the number of pairs $(\lambda_*,\varrho_*)$ satisfying condition \ref{item:trim_proof2}, is $d(n-d)$. Therefore, \[|\mathcal J_{\Theta(\TT^\times)}|=\textstyle\binom{d}{2}+\binom{n-d}{2}+d(n-d)=\binom{n}{2}.\] 

We are left to show that $\mathrm{height}(\Theta(\TT^\times))=\binom{n}{2}$. Since the height of a lattice is always at most the number of join-irreducible elements of the lattice, it suffices to construct a chain in $\Theta(\TT^\times)$ of length $\binom{n}{2}$. The maximal element of $\LL(\TT^\times)$ is the permutation $\lambda_{\max}=(d+1)(d+2)\cdots n\,12\cdots d$, which has $d(n-d)$ inversions. Therefore, $\LL(\TT^\times)$ has a maximal chain $\lambda_0\lessdot\lambda_1\lessdot\cdots\lessdot\lambda_{d(n-d)}$ of length $d(n-d)$, where $\lambda_0=12\cdots n$ and $\lambda_{d(n-d)}=\lambda_{\max}$. Note that 
\begin{equation}\label{eq:chain1}(\lambda_0,\varrho_{\min})<(\lambda_1,\varrho_{\min})<\cdots<(\lambda_{d(n-d)},\varrho_{\min})
\end{equation} is a chain in $\Theta(\TT^\times)$. We have \[\mathrm{height}(\OO(\TT^\times))=\mathrm{height}(\OO(\TT_1)\times\OO(\TT_{d+1}))=\mathrm{height}(\OO(\TT_1))+\mathrm{height}(\OO(\TT_{d+1}))=\textstyle \binom{d}{2}+\binom{n-d}{2},\] where the last equality follows from the fact that the $m$-th Tamari lattice has height $\binom{m}{2}$. Hence, $\OO(\TT^\times)$ has a maximal chain $\varrho_0\lessdot\varrho_1\lessdot\cdots\lessdot\varrho_{M}$, where $\varrho_0=\varrho_{\min}$ and $M=\binom{d}{2}+\binom{n-d}{2}=\binom{n}{2}-d(n-d)$. For every $\varrho\in\OO(\TT^\times)$, the pair $(\lambda_{\max},\varrho)$ is in $\Theta(\TT^\times)$. 
Therefore,  
\begin{equation}\label{eq:chain2}(\lambda_{\max},\varrho_0)<(\lambda_{\max},\varrho_1)<\cdots<(\lambda_{\max},\varrho_{M})
\end{equation}
is a chain in $\Theta(\TT^\times)$. By concatenating the chains in \eqref{eq:chain1} and \eqref{eq:chain2}, we obtain a chain in $\Theta(\TT^\times)$ of length $\binom{n}{2}$. 
\end{proof}

\section{Stacks and Brooms}\label{sec:stacks}

Let $\mathcal W$ denote the set of finite words over the alphabet of positive integers in which no letter appears more than once. West's \dfn{stack-sorting map} is the function $\s\colon\mathcal W\to\mathcal W$ defined recursively as follows.\footnote{The stack-sorting map can also be defined via a procedure that sends a word through a \emph{stack} in a right-greedy manner. Alternatively, it can be defined using postorder and in-order traversals of decreasing binary plane trees. See, e.g., \cites{Bona_survey_2019,DefantThesis}.} As a base case, we define $\s(\epsilon)=\epsilon$, where $\epsilon$ is the empty word. Now, if $\sigma\in\mathcal W$ is nonempty, then we can write $\sigma=\mathsf{L}n\mathsf{R}$, where $n$ is the largest letter in $\sigma$. With this notation, we define $\s(\sigma)=\s(\mathsf{L})\s(\mathsf{R})n$. For example, \[\s(316452)=\s(31)\,\s(452)\,6=\s(1)\,3\,\s(4)\,\s(2)\,56=134256.\] We usually restrict the stack-sorting map to the symmetric group $\mathfrak S_n$ and view it as a function $\s\colon\mathfrak S_n\to\mathfrak S_n$. 

The following lemma is well known and follows readily from the definition of $\s$. 

\begin{lemma}\label{lem:stack_inversions}
Let $1\leq a<b\leq n$. For $\sigma\in\mathfrak S_n$, we have $(a,b)\in\Inv(\s(\sigma))$ if and only if there exists $c\in[n]$ such that $b<c$ and $b\prec_{\sigma} c\prec_{\sigma} a$.
\end{lemma}

Recall the definition of the broom $\Broom_{k,n}$ from \cref{sec:intro}. Let \[\wo(k,n)=k(k-1)\cdots 1(k+1)(k+2)\cdots n\in \mathfrak S_n.\] Note that 
\begin{equation}\label{eq:Delta_inversions}
\Delta_{\Weak(\mathfrak S_n)}(\wo(k,n))=\{w\in\mathfrak S_n:j\leq k\text{ for all }(i,j)\in\Inv(w)\}.
\end{equation}
Our goal in this section is to prove \cref{thm:stack-sorting}, which states that $\MN(\Broom_{k,n})$ is isomorphic to the subposet $\s^{-1}(\Delta_{\Weak(\mathfrak S_n)}(\wo(k,n)))$ of $\Weak(\mathfrak S_n)$. 

As usual, let us identify the vertex set of $\Broom_{k,n}$ with $\{0,1,\ldots,n\}$ so that $0,1,\ldots,n$ is the preorder traversal. Consider a pair $(\lambda,\varrho)\in\Theta(\Broom_{k,n}^\times)$. Let us write $\lambda=w_1\cdots w_n$. For each $u\in[n]$, let $A_\varrho(u)=\{j\in[n]:u\in\varrho(n+1-j)\}$. Let $B_{(\lambda,\varrho)}(w_\ell)=A_{\varrho}(w_\ell)\setminus\bigcup_{i=\ell+1}^{n}A_{\varrho}(w_i)$. Let us write $\mu_{(\lambda,\varrho)}(w_\ell)$ for the word obtained by reading the elements of $B_{(\lambda,\varrho)}(w_\ell)$ in decreasing order. Finally, let \[\Omega(\lambda,\varrho)=\mu_{(\lambda,\varrho)}(w_n)\mu_{(\lambda,\varrho)}(w_{n-1})\cdots \mu_{(\lambda,\varrho)}(w_1)\in \mathfrak S_n.\] 

\begin{example}
Let $k=4$ and $n=9$. Let $\lambda=123459867$, and let $\varrho$ be the ornamentation of $\Broom_{4,9}^\times$ depicted in \cref{fig:ornamented_broom}. Note that $(\lambda,\varrho)\in\Theta(\Broom_{4,9}^\times)$. We have 
\begin{alignat*}{3}
&A_\varrho(1)=\{9\}, &&A_{\varrho}(2)=\{8,9\}, &&A_{\varrho}(3)=\{7\}, \\
&A_\varrho(4)=\{6,7\},\quad &&A_{\varrho}(5)=\{5,6,7\},\quad &&A_{\varrho}(6)=\{4,7\}, \\ 
&A_\varrho(7)=\{3\}, &&A_{\varrho}(8)=\{2,6,7\}, &&A_{\varrho}(9)=\{1,6,7\},
\end{alignat*} so 
\begin{alignat*}{3}
&B_{(\lambda,\varrho)}(w_9)=B_{(\lambda,\varrho)}(7)=\{3\},\quad &&B_{(\lambda,\varrho)}(w_8)=B_{(\lambda,\varrho)}(6)=\{4,7\},\quad &&B_{(\lambda,\varrho)}(w_7)=B_{(\lambda,\varrho)}(8)=\{2,6\}, \\ 
&B_{(\lambda,\varrho)}(w_6)=B_{(\lambda,\varrho)}(9)=\{1\}, &&B_{(\lambda,\varrho)}(w_5)=B_{(\lambda,\varrho)}(5)=\{5\}, &&B_{(\lambda,\varrho)}(w_4)=B_{(\lambda,\varrho)}(4)=\emptyset, \\
&B_{(\lambda,\varrho)}(w_3)=B_{(\lambda,\varrho)}(3)=\emptyset, &&B_{(\lambda,\varrho)}(w_2)=B_{(\lambda,\varrho)}(2)=\{8,9\}, &&B_{(\lambda,\varrho)}(w_1)=B_{(\lambda,\varrho)}(1)=\emptyset. 
\end{alignat*}
Therefore, $\Omega(\lambda,\varrho)=374621598$. 
\end{example}

\begin{figure}[ht]
\begin{center}\includegraphics[height=7.490cm]{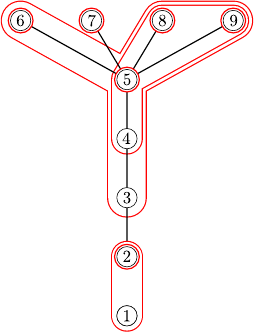}
  \end{center}
  \caption{An ornamentation of $\Broom_{4,9}^\times$.}\label{fig:ornamented_broom}
\end{figure}

\begin{lemma}\label{lem:Omega_inversions}
Fix positive integers $k\leq n$, and let $(i,j)$ be a pair such that $1\leq i<j\leq n$. Let $(\lambda,\varrho)\in\Theta(\Broom_{k,n}^\times)$. If $j\leq k$, then $(i,j)\in\Inv(\Omega(\lambda,\varrho))$ if and only if $(n+1-j,n+1-i)\in\Inv(\lambda)$. If $j\geq k+1$, then $(i,j)\in\Inv(\Omega(\lambda,\varrho))$ if and only if $n+1-i\in\varrho(n+1-j)$. 
\end{lemma}

\begin{proof}
Let us write $\lambda=w_1\cdots w_n$. Note that $w_r=r$ for $1\leq r\leq n-k$ and that \[{\{w_{n-k+1},\ldots,w_{n}\}=\{n-k+1,\ldots,n\}}.\] For $m\in[n]$, let 
\[
\ell(m)=\max\{s\in[n]:w_s\in\varrho(n+1-m)\}.
\]
Then $\ell(m)$ is the unique element of $[n]$ such that $m\in B_{(\lambda,\varrho)}(w_{\ell(m)})$. It follows that $(i,j)$ is an inversion of $\Omega(\lambda,\varrho)$ if and only if $\ell(i)\leq\ell(j)$. 

If $j\leq k$, then $\varrho(n+1-i)=\{n+1-i\}$ and $\varrho(n+1-j)=\{n+1-j\}$, so $w_{\ell(i)}=n+1-i$ and $w_{\ell(j)}=n+1-j$. In this case, $\ell(i)\leq\ell(j)$ if and only if $(n+1-j,n+1-i)\in\Inv(\lambda)$. 

Now suppose that $j\geq k+1$. Then $n+1-j\leq_{\Broom_{k,n}} n+1-i$, so either $n+1-i\in\varrho(n+1-j)$ or $\varrho(n+1-j)\prec_{\lambda}\varrho(n+1-i)$. Hence, $\ell(i)\leq\ell(j)$ if and only if $n+1-i\in\varrho(n+1-j)$. 
\end{proof}

\begin{lemma}\label{lem:Omega_image}
Fix positive integers $k\leq n$. If $(\lambda,\varrho)\in\Theta(\Broom_{k,n}^\times)$, then \[\s(\Omega(\lambda,\varrho))\in\Delta_{\Weak(\mathfrak S_n)}(\wo(k,n)).\]     
\end{lemma}

\begin{proof}
In light of \eqref{eq:Delta_inversions}, we must show that every inversion $(a,b)$ of $\s(\Omega(\lambda,\varrho))$ satisfies $b\leq k$. Thus, assume by way of contradiction that there exists $(a,b)\in\Inv(\s(\Omega(\lambda,\varrho)))$ with $b\geq k+1$. According to \cref{lem:stack_inversions}, there exists $c\in[n]$ such that $b<c$ and $b\prec_{\Omega(\lambda,\varrho)}c\prec_{\Omega(\lambda,\varrho)}a$. Let $\ell(b)$ and $\ell(c)$ be the unique indices such that $b\in B_{(\lambda,\varrho)}(w_{\ell(b)})$ and $c\in B_{(\lambda,\varrho)}(w_{\ell(c)})$. The pairs $(a,b)$ and $(a,c)$ are both inversions of $\Omega(\lambda,\varrho)$, so we can appeal to \cref{lem:Omega_inversions} to find that $n+1-a\in\varrho(n+1-b)$ and $n+1-a\in\varrho(n+1-c)$. This implies that the ornaments $\varrho(n+1-b)$ and $\varrho(n+1-c)$ are nested; since $n+1-c\leq_{\Broom_{k,n}} n+1-b$, we must have $\varrho(n+1-b)\subseteq\varrho(n+1-c)$. It follows that \[\{u\in[n]:b\in A_\varrho(u)\}\subseteq \{u\in[n]:c\in A_\varrho(u)\}.\] Consequently, $\ell(b)\leq\ell(c)$. This implies that $c\prec_{\Omega(\lambda,\varrho)} b$, which is our desired contradiction. 
\end{proof}

\cref{lem:Omega_image} tells us that we actually have a map $\Omega\colon\Theta(\Broom_{k,n}^\times)\to\s^{-1}(\Delta_{\Weak(\mathfrak S_n)}(\wo(k,n)))$. In light of \cref{prop:isomorphism}, the following proposition implies \cref{thm:stack-sorting}. 

\begin{proposition}
Fix positive integers $k\leq n$. The map \[\Omega\colon\Theta(\Broom_{k,n}^\times)\to\s^{-1}(\Delta_{\Weak(\mathfrak S_n)}(\wo(k,n)))\] is a poset isomorphism.  
\end{proposition}

\begin{proof}
Let us first argue that $\Omega$ is a bijection. It is immediate from \cref{lem:Omega_inversions} that $\Omega$ is injective. To prove surjectivity, let us choose an arbitrary $\sigma=\sigma(1)\cdots\sigma(n)\in\s^{-1}(\Delta_{\Weak(\mathfrak S_n)}(\wo(k,n)))$. Let $v_1,\ldots,v_k$ be the elements of $\{n-k+1,\ldots,n\}$, listed so that 
\begin{equation}\label{eq:listed}
n+1-v_k\prec_{\sigma}\cdots\prec_\sigma n+1-v_1.
\end{equation} 
Let $\lambda=12\cdots(n-k)v_1\cdots v_k$. Define an ornamentation $\varrho$ of $\Broom_{k,n}^\times$ as follows. For $1\leq j \leq k$, we must define $\varrho(n+1-j)=\{n+1-j\}$. For $k+1\leq j\leq n$, let \[\varrho(n+1-j)=\{n+1-j\}\cup\{n+1-i:(i,j)\in\Inv(\sigma)\}.\] 

Let us first show that the map $\varrho$ is indeed an ornamentation. Let $j\in[n]$; we will prove that $\varrho(n+1-j)$ induces a connected subgraph of $\Broom_{k,n}$. If $1\leq j\leq k$, then this is obvious because $\varrho(n+1-j)=\{n+1-j\}$. Now assume $k+1\leq j\leq n$. Suppose $i$ and $i'$ are vertices satisfying 
\begin{equation}\label{eq:jii'}
n+1-j<_{\Broom_{k,n}} n+1-i<_{\Broom_{k,n}} n+1-i'
\end{equation} and $n+1-i'\in\varrho(n+1-j)$; we must show that $n+1-i\in\varrho(n+1-j)$. The fact that $n+1-i'\in\varrho(n+1-j)$ implies that $(i',j)\in\Inv(\sigma)$. It follows from \eqref{eq:jii'} that $i'<i<j$ and that $i\geq k+1$. Since $\s(\sigma)\in\Delta_{\Weak(\mathfrak S_n)}(\wo(k,n))$, we know that $(i',i)\not\in\Inv(\s(\sigma))$. According to \cref{lem:stack_inversions}, we cannot have $i\prec_\sigma j\prec_\sigma i'$. But we know that $j\prec_{\sigma}i'$, so we must have $j\prec_{\sigma} i$. This shows that $(i,j)\in\Inv(\sigma)$, so $n+1-i\in\varrho(n+1-j)$, as desired. 

We now must show that for all vertices $j,j'\in[n]$, the sets $\varrho(j)$ and $\varrho(j')$ are either nested or disjoint. Suppose $j,j'\in[n]$ are such that $j\leq j'$ and $\varrho(n+1-j)\cap\varrho(n+1-j')\neq\emptyset$; we will show that $\varrho(n+1-j)\subseteq\varrho(n+1-j')$. If $1\leq j\leq k$, then this is immediate since $\varrho(n+1-j)$ is a singleton set. Now suppose $j\geq k+1$. Then $n+1-j\in\varrho(n+1-j')$, so $(j,j')\in\Inv(\sigma)$. It follows that if $i\in[n]$ is such that $(i,j)\in\Inv(\sigma)$, then $(i,j')\in\Inv(\sigma)$. Hence, $\varrho(n+1-j)\subseteq\varrho(n+1-j')$. This completes the proof that $\varrho$ is an ornamentation. 

If we can show that $(\lambda,\varrho)\in\Theta(\Broom_{k,n}^\times)$, then it will follow from \cref{lem:Omega_inversions} that $\Omega(\lambda,\varrho)=\sigma$, which will prove that $\Omega$ is surjective. Thus, we must show that for every $j\in[n]$, the elements of $\varrho(n+1-j)$ form a consecutive factor of $\lambda$. If $1\leq j\leq k$, then this is trivial since $\varrho(n+1-j)$ is a singleton set, so assume $k+1\leq j\leq n$. Referring to the definition of $\lambda$, we see that we must demonstrate that the elements of the set $\Gamma=\{n+1-i:1\leq i\leq k\text{ and }(i,j)\in\Inv(\sigma)\}$ form a prefix of the word $v_1\cdots v_k$. Thus, suppose $1\leq i<i'\leq k$ and $v_{i'}\in\Gamma$. Then $(n+1-v_{i'},j)\in\Inv(\sigma)$, so $j\prec_\sigma n+1-v_{i'}\prec_{\sigma} n+1-v_i$ (by \eqref{eq:listed}). It follows that $(n+1-v_i,j)\in\Inv(\sigma)$, so $v_i\in\Gamma$. 

We have shown that $\Omega$ is bijective. It is a straightforward consequence of \cref{lem:Omega_inversions} that $\Omega$ and $\Omega^{-1}$ are order-preserving. 
\end{proof}

\section{Future Directions}\label{sec:conclusion}

In \cref{thm:semidistributive,thm:trim}, we characterized the operahedron lattices that are semidistributive and the operahedron lattices that are trim. We also mentioned in \cref{rem:semidistrim} that an operahedron lattice is semidistrim if and only if it is semidistributive. It is natural to investigate other structural properties of operahedron lattices. 

As mentioned in \cref{sec:intro}, operahedra are special examples of poset associahedra. Laplante-Anfossi~\cite{laplante2022diagonal} also noted that they are the \emph{graph associahedra} of the line graphs of trees. Graph associahedra were introduced by Carr and Devadoss \cite{CarrDevadoss} and further popularized by Postnikov \cite{Postnikov_GP} as examples of \emph{generalized permutohedra}. It would be interesting to further investigate which poset associahedra or graph associahedra have $1$-skeletons that can be oriented in some natural way to produce lattices (say, using the realizations in \cite{sack2023realization} or \cite{devadoss2009realization}). Barnard and McConville already have work along these lines \cite{BarnardMcConville}, but there are further avenues worth pursuing. For example, one could consider graph associahedra of particular families of graphs such as block graphs or chordal graphs. 

In order to understand the operahedron lattice of a tree $\TT$, we first had to introduce the ornamentation lattice $\OO(\TT^\times)$. While ornamentation lattices are generally less complicated than operahedron lattices, it could still be interesting to consider ornamentation lattices in their own right by asking more refined questions than those asked here about operahedron lattices. For example, ornamentation lattices could provide a fruitful landscape for generalizing results about Tamari lattices---which are ornamentation lattices of chains---such as those in \cites{BarnardComplex,Bostan,Clement,Ceballos1,ChatelPons,DefantMeeting,FishelNelson,Hong}.

It is natural to ask if there are even stronger connections between operahedron lattices and stack-sorting. In \cref{thm:stack-sorting}, we found an isomorphism between the operahedron lattice of a broom and the subposet of $\Weak(\mathfrak S_{n})$ consisting of the stack-sorting preimages of a certain set of permutations. Are there families of trees more general than brooms for which similar isomorphisms exist?

\section*{Acknowledgments}
Colin Defant was supported by the National Science Foundation under Award No.\ 2201907 and by a Benjamin Peirce Fellowship at Harvard University. Andrew Sack was supported by the National Science Foundation
Graduate Research Fellowship Program under Grant No. DGE-2034835 and National Science Foundation Grants No. DMS-1954121 and DMS-2046915. Any
opinions, findings, and conclusions or recommendations expressed in this material
are those of the authors and do not necessarily reflect the views of the National
Science Foundation.

\bibliographystyle{amsalpha}
\bibliography{main}

\end{document}